\documentclass[a4paper]{article}
\usepackage{amssymb}
\usepackage{amsmath}
\usepackage{amsthm}
\usepackage{authblk}
\usepackage[USenglish]{babel}
\usepackage{graphicx}
\usepackage[colorlinks=true]{hyperref}
\usepackage[latin1]{inputenc}
\newtheorem{theorem}{Theorem}[section]
\newtheorem{lemma}[theorem]{Lemma}
\newtheorem{proposition}[theorem]{Proposition}
\newtheorem{corollary}[theorem]{Corollary}
\newtheorem{remark}[theorem]{Remark}
\newtheorem{definition}[theorem]{Definition}
\newtheorem{example}[theorem]{Example}
\numberwithin{equation}{section}
\newcommand{\N}{\mathbb N}

\newcommand{\R}{\mathbb R}

\newcommand{\I}{\mathbb I}

\newcommand{\mbf}{\mathbf}
\newcommand{\mcal}{\mathcal}

\newcommand{\mrm}{\mathrm}

\renewcommand{\a}{\alpha}

\newcommand{\G}{\Gamma}
\renewcommand{\d}{\delta}
\newcommand{\D}{\Delta}
\newcommand{\e}{\varepsilon}

\renewcommand{\t}{\theta}
\newcommand{\Th}{\Theta}

\renewcommand{\O}{\Omega}
\newcommand{\wt}{\widetilde}
\newcommand{\wh}{\widehat}
\newcommand{\ol}{\overline}

\newcommand{\pa}{\partial}
\newcommand{\n}{\nabla}
\newcommand{\fa}{\forall}

\newcommand{\us}{\underset}

\newcommand{\To}{\Rightarrow}
\newcommand{\map}{\mapsto}

\newcommand{\sm}{\setminus}

\newcommand{\sub}{\subset}

\newcommand{\eq}{\equiv}

\newcommand{\x}{\times}
\renewcommand{\c}{\circ}
\newcommand{\q}{\quad}
\renewcommand{\l}{\left}
\renewcommand{\r}{\right}
\newcommand{\bin}[2]{\left(\genfrac{}{}{0pt}{}{#1}{#2}\right)}
\newcommand{\qm}[1]{``#1''}
\allowdisplaybreaks
\begin{document}

\everymath{\displaystyle}

\title{Non-uniqueness of blowing-up solutions to the Gelfand problem}
\author{Luca Battaglia\thanks{Universit\`a degli Studi Roma Tre, Dipartimento di Matematica e Fisica, Largo S. Leonardo Murialdo 1, 00146 Roma - lbattaglia@mat.uniroma3.it}, Massimo Grossi\thanks{Sapienza Universit\`a di Roma, Dipartimento di Matematica, Piazzale Aldo Moro 5, 00185 Roma - massimo.grossi@uniroma1.it}, Angela Pistoia\thanks{Sapienza Universit\`a di Roma, Dipartimento di Scienze di Base e Applicate, Via Antonio Scarpa 16, 00161 Roma - angela.pistoia@uniroma1.it}}
\date{}

\maketitle\

\begin{abstract}
\noindent
We consider the \emph{Gelfand problem}
$$\l\{\begin{array}{ll}-\D u=\rho^2V(x)e^u&\text{in }\O\\u=0&\text{on }\pa\O\end{array}\r.,$$
where $\O$ is a planar domain and $\rho$ is a positive small parameter.\\
Under some conditions on the potential $0<V\in C^\infty\l(\ol\O\r)$, we provide the first examples of multiplicity for blowing-up solutions at a given point in $\O$ as $\rho\to0.$
The argument is based on a refined Lyapunov-Schmidt reduction and the computation of the degree of a finite-dimensional map.
\end{abstract}\

\section{Introduction}\

We consider the following problem, known as \emph{Gelfand problem}:
\begin{equation}\label{liou}
\l\{\begin{array}{ll}-\D u=\rho^2V(x)e^u&\text{in }\O\\u=0&\text{on }\pa\O\end{array}\r.;
\end{equation}
where $\O\sub\R^2$ is a smooth bounded domain, $\rho>0$ is a positive parameter and $V(x)\in C^\infty\l(\ol\O\r)$ is a smooth positive function.\\
Such an equation has been intensively studied in the recent decades due to its many applications in different fields, such as Gaussian curvature prescription problem in conformal geometry (see for instance \cite{kw,cy1,cy2}), Chern-Simons theory in mathematical physics (see \cite{tar,yang}) and description of Euler flow in statistical mechanics (see \cite{clmp1,clmp2,kie}). In nonlinear analysis, it is considered a \emph{critical} nonlinearity for planar elliptic problems, a counterpart of the higher-dimensional critical Sobolev equation.\\

Several results have been given concerning existence and multiplicity of solutions to \eqref{liou}, both using variational methods (\cite{dem,dm}) and computing the Leray-Schauder degree (\cite{chenlin}). A rather complete blow-up analysis has also been provided by different authors (\cite{bremer,ls,mw,chenlin,nasu}).\\
In case of a blowing-up family of solutions to \eqref{liou} as $\rho$ goes to $0$ with finite mass, blow-up occurs at a finite number of distinct internal points $\xi^1,\dots,\xi^N\in\O$ with $\xi^i\ne\xi^j$ for $i\ne j$, with no residual mass; moreover, the $N$-tuple of concentration points is a critical point of the \emph{reduced functional}
\begin{equation}\label{fxin}
\mcal F\l(\xi^1,\dots,\xi^N\r):=\sum_{i=1}^NH\l(\xi^i,\xi^i\r)+\sum_{i,j=1,i\ne j}^NG\l(\xi^i,\xi^j\r)+\frac1{4\pi}\sum_{i=1}^N\log V\l(\xi^i\r),
\end{equation}
where $G$ denotes Green's function of $-\D$ on $\O$ and $H$ its regular part, namely
$$\l\{\begin{array}{ll}-\D G(x,y)=\d_y(x)&x\in\O\sm\{y\}\\G(x,y)=0&x\in\pa\O\end{array}\r.\q\q\q\q\q\q H(x,y):=G(x,y)+\frac1{2\pi}\log|x-y|.$$\

As a counterpart of this blow-up analysis, in \cite{bp,dkm,egp} families of blowing-up solutions to \eqref{liou} have been constructed, with the concentration points being any stable critical points of \eqref{fxin}.\\
If $N=1$ the function \eqref{fxin} reduces to 
\begin{equation}\label{f}
\mcal F(\xi)=H(\xi,\xi)+\frac1{4\pi}\log V(\xi),\q\xi\in\O 
\end{equation}
which has always a critical point, i.e. a maximum point. If $N\ge2$ 
critical points of $\mcal F$ always exist if $\O$ is multiply connected (see \cite{dkm}) or if it is a dumbbell-shaped domain (see \cite{egp}), whereas if $\O$ is convex they do not exist at all (see \cite{tg}).\\

Uniqueness of blowing-up families has been addressed in \cite{bjly1,bjly2}. In these papers, the authors prove that a sequence of solutions to \eqref{liou} blowing-up at given $\xi^1,\dots,\xi^N$ is unique provided the point $\l(\xi^1,\dots,\xi^N\r)$ is a non-degenerate critical point of $\mathcal F$, namely $D^2\mcal F\l(\xi^1,\dots,\xi^N\r)$ is invertible. A similar result had already been proved in \cite{gg} for domains being symmetric and convex with respect to both axes. Their results can be summarized in the following 

\begin{theorem}
(\cite{bjly1,bjly2,gg})$ $\\
Let $V\in C^\infty\l(\ol\O\r)$ be a positive potential such that the energy functional $\mcal F$ defined by \eqref{fxin} has a non-degenerate critical point $\l(\xi^1,\dots,\xi^N\r)$, i.e.
\begin{equation}\label{nond}
\det\l(D^2\mcal F\l(\xi^1,\dots,\xi^N\r)\r)\ne0.
\end{equation}
Let $u_\rho^1,u_\rho^2$ be two families of solutions blowing-up at the same $\xi^1,\dots,\xi^N$ as $\rho$ goes to $0$.
Then, for small $\rho$, one has $u_\rho^1\eq u_\rho^2$.
\end{theorem}

Let us point out that the non-degeneracy condition \eqref{nond} is \qm{almost always} satisfied. Indeed, it is clear that for generic potential $V$ the function $\mathcal F$ is a Morse function. On the other hand, if $V\eq1$ for generic domains $\Omega$ the function $\mathcal F$ is still a Morse function as proved in \cite{mipi,bamipi}. Therefore, roughly speaking we could say that solutions blowing-up at a given critical point of $\mathcal F$ are \qm{almost always} unique. Hence it is quite natural to ask if the uniqueness does still holds when the blow-up point is a \emph{degenerate} critical point of $\mcal F$. The aim of this paper is to build examples of degenerate critical points of $\mcal F$ for which the uniqueness does not hold anymore.\\

In order to state our main result, it is necessary to introduce some notations.\\
Without loss of generality, we assume that $0$ is a critical point of the function $\mathcal F$ defined in \eqref{f}. We believe that the same argument may work also for blow-up at multiple points, but we will consider only one point in order to simplify computations and notations. Precisely, we look for solutions to \eqref{liou} blowing-up in the following sense:

\begin{definition}\label{blowup}$ $\\
Let $\{u_\rho\}$ be a family of solution to \eqref{liou} for some $\rho\to0$. We say that $u_\rho$ \emph{blows-up} at $0\in\O$ if the following occur:
\begin{itemize}
\item $u_\rho$ is uniformly bounded from above in $L^\infty_{\mrm{loc}}\l(\ol\O\sm\{0\}\r)$;
\item There exists a sequence $\xi_\rho\us{\rho\to0}\to0$ such that $u_\rho(\xi_\rho)\us{\rho\to0}\to+\infty$.
\end{itemize}
\end{definition}

We want to build two distinct solutions $u_\rho^1\not\equiv u_\rho^2$ to problem \eqref{liou} such that $u_\rho^1\l(\xi_\rho^1\r),u_\rho^2\l(\xi_\rho^2\r)\to+\infty$ for a suitable choice of distinct points $\xi_\rho^1\ne\xi_\rho^2$, both approaching $0$ as $\rho\to0$.\\
Next let us describe the assumptions on $\mcal F$: we choose its second derivatives to be all vanishing at $0$ and its third derivative to satisfy some non-degeneracy condition at $0$, in a sense described by the following definition.

\begin{definition}\label{ammis}$ $\\
Let $V\in C^\infty\l(\ol\O\r)$ be a positive potential.
We say that $V$ is \emph{admissible} if the functional $\mcal F$ defined by \eqref{f} satisfies the following properties:
\begin{itemize}
\item $\pa_{\xi_i}\mcal F(0)=\pa^2_{\xi_i\xi_j}\mcal F(0)=0$ for all $i,j=1,2$;
\item The homogeneous polynomial map $\mcal P:\mathbb R^2\to\mathbb R$ defined as
\begin{equation}\label{p}
\mcal P(\xi)=\frac{4\pi^2}3\l\langle D^3\mcal F(0),\xi,\xi,\xi\r\rangle=\frac{4\pi^2}3\sum_{i,j,k=1}^2\l(\pa^3_{\xi_i\xi_j\xi_k}\mcal F(0)\r)\xi_i\xi_j\xi_k
\end{equation}
has $\xi=0$ as its only critical point.
\end{itemize}
\end{definition}

The main result of this paper reads as follows.

\begin{theorem}\label{molt}$ $\\
Let $V\in C^\infty\l(\ol\O\r)$ be a positive admissible potential (in the sense of Definition \ref{ammis}). Set
\begin{equation}\label{eta}
\eta_0:=\Big(64\pi^3\l(\pa^2_{x_1y_1}H(0,0)+\pa^2_{x_2y_2}H(0,0)\r)\n_xH(0,0)-\pi\n(\D\log V)(0)\Big)\frac{V(0)}8e^{8\pi H(0,0)}.
\end{equation}
If the equation
\begin{equation}\label{sc}\n\mcal P(\xi)=\eta_0
\end{equation}
has $K$ distinct stable solutions (in the sense of Definition \ref{stable}), then there exist $\rho_0>0$ and $K$ families of solutions $\l\{u_\rho^i\r\}$, $i=1,\dots,K$, to \eqref{liou} for $\rho\in(0,\rho_0)$, all blowing-up at $0$ as $\rho$ goes to $0$ (in the sense of Definition \ref{blowup}) and such that $u_\rho^i\not\equiv u_\rho^j$ if $i\not=j$.
\end{theorem}\

We will use the following definition of stable solution.\\
\begin{definition}\label{stable}$ $\\
Let $\mcal Z:\R^n\to\R^n$ be a continuous function. We say that $x_0$ is a stable solution to the equation $\mcal Z(x)=y$ if for any $\e>0$ small enough and $\mcal W:\R^n\to\R^n$ with $\|\mcal W-\mcal Z\|_{C^0(\R^n)}\le\e$ there exists $x_\e\in\R^n$ such that $\mcal W(x_\e)=y$ and $x_\e\to x_0$ as $\e\to 0$ 
\end{definition}
In particular, if $x_0$ is the only solution to $\mcal Z(x_0)=y$ in $B_R(x_0)$ and the Brouwer degree $\deg(\mcal Z,B_R(x_0),y)$ is different from zero, then $x_0$ is a stable solution to the equation $\mcal Z(x)=y.$

\begin{example}
$ $\\
We point out that it is always possible to choose the potential $V$ so that all the assumptions of Theorem \ref{molt} hold true.
Let $V$ be such that in a neighborhood of the origin the following Taylor expansion for the function $\mcal F$ holds true
$$\mcal F(\xi)=H(\xi,\xi)+\frac1{4\pi}\log V(\xi)=\a\xi_1^3-\xi_1\xi_2^2,\q\text{with }\a>0.$$
It is clear that $0$ is a fully degenerate critical point of $\mcal F$. Moreover, a simple computation shows that
$$\mcal P(\xi)=8\pi^2\l(\a\xi_1^3-\xi_1\xi_2^2\r).$$
We remark that $\mcal P$ has only one critical point, namely the origin. Moreover, the vector $\eta_0$ defined in \eqref{eta} depends on $\a$ and it vanishes for a unique choice of $\a=\a_0$. A direct computation shows that the equation
$\nabla\mcal P(\xi)=\eta_0(\a)$ has two stable solutions for $\a\ne\a_0$.
\end{example}\

\begin{remark}$ $\\
Theorem \ref{molt} deals with the case when the order of degeneracy of the function $\mcal F$ at the origin is $3$ in the sense of Definition \ref{ammis}. If $\mcal F$ has an higher order of degeneracy (as in Definition \ref{ammis2}) the situation becomes more delicate. Indeed, in this case the vector $\eta_0$ defined in \eqref{eta} does not depend on the potential $V$ and, moreover, if the domain $\O$ is simply connected the vector $\eta_0$ is surprisingly equal to zero (see Lemma \ref{semplcon}). On the other hand, if $\O$ is multiply connected, $\eta_0$ can be different from zero and multiplicity of solutions blowing-up at the same point can be proved (see Theorem \ref{molt2}). This case will be studied in detail in Section 6.
\end{remark}\

Actually, the number of solutions blowing-up at one point is strongly related with the number of solutions of the equation \eqref{sc}, namely if we want more than one solution to \eqref{liou} we need multiple solutions to the equation \eqref{sc}. We point out that if $\eta_0=0$ the equation $\n\mcal P(\xi)=0$ has a unique solution because of Definition \ref{ammis}. Then, we need to assume $\eta_0\not=0$ and suitable conditions on $\mcal P $ which guarantee the existence of multiple solutions.\\
A classical tool to compute the number of solutions of a finite-dimensional equation is the topological degree: if the degree in absolute value is greater or equal than $2$, then multiplicity of solutions is ensured. In Proposition \ref{grado} we will compute the degree of $\n\mcal P$ via the number of nodal lines of $\mcal P$, a result which we believe is also interesting in itself. In particular, if $\mcal P$ has three distinct lines of zeros then equation \eqref{sc} has two distinct solutions.\\
Clearly one may have multiplicity of solutions even if the degree is $0$ or $\pm1$, but in this case this may depend on the constant term $\eta_0$, and it can also be more difficult to be verified.\\
The computation of the degree of $\n\mcal P$ gives the following corollary to Theorem \ref{molt}.

\begin{corollary}$ $\\
Assume the polynomial map $\mcal P$ (see Definition \ref{ammis}) has three different nodal lines and $\eta_0\ne0$.\\
Then, there exist two families of solutions $\l\{u_\rho^1\r\},\l\{u_\rho^2\r\}$ to \eqref{liou} for $\rho\in(0,\rho_0)$, both blowing-up at $0$ as $\rho$ goes to $0$ and such that $u_\rho^1\not\equiv u_\rho^2$
\end{corollary}\

The homogeneous polynomial map $\mcal P$ associated to the third derivatives of $\mcal F$ plays a crucial role, much like $\mcal F$ itself did in the original construction in \cite{dkm,egp}, which uses similar techniques to the present paper. Indeed, in those papers blowing-up solutions to \eqref{liou} have been constructed using a \emph{Lyapunov-Schmidt reduction}: once the main order term $\mrm PU$ of the solution is prescribed as the projection of the \emph{bubble} centered at some $ \xi_\rho\in\O$ having the profile of entire solutions to Liouville equations (see \eqref{u}), we find a suitable remainder $ \psi_\rho$ satisfying $\|\psi_\rho\|_{H^1_0(\O)}\to0$ as ${\rho\to0}$ in such a way that $u=\mrm PU+\psi_\rho$ solves \eqref{liou}. To this purpose, the choice of the point $\xi_\rho$ is crucial. In particular, $\xi_\rho\to0$ as $\rho\to0$ and $\n\mcal F(0)=0$.\\
In this paper, our goal is to find a second order condition to be satisfied by the point $\xi_\rho$. More precisely, we show that $\xi_\rho=\xi_0\rho\sqrt{\log\frac1\rho}$ and $\xi_0$ solves the equation \eqref{sc}.\\
In order to find the second order condition \eqref{sc}, we need to improve the first order approximation term of the solution adding two higher order correction terms. More precisely, we will look for a solution like
$$u(x)=\mrm PU(x)+\rho^2\tau^2\l(\mrm P\wh w\l(\frac{x-\xi}{\rho\tau}\r)+\wt W(x)\r)+\phi_\rho(x),$$
where the first order term $\mrm PU$ is as usual the projection of the standard bubble (see definitions \eqref{u} and \eqref{pr}). The refinement of our ansatz is given by the new functions $\mrm P\wh w$ (introduced in Subsection \ref{sec_pw}) and $\wt W$ (introduced in Subsection \ref{sec_w}) which give a \emph{local} and a \emph{global} correction, respectively. The remainder $\phi_\rho$ is much smaller than the previous remainder term $\psi_\rho$, hence it can be ignored in computing the equation for $\xi_\rho$. On the other hand, the correction terms $\mrm P\wh w$ and $\wt W$ are not negligible and they originate the vector $\eta_0$ in the equation, as we will see in the following sections.\\

We point out that our result is inspired by non-uniqueness results obtained in \cite{gro,gn} for the Schr\"odinger equation on the whole space and on bounded domains with Neumann conditions, where a second order expansion of the concentration point is performed. However, in those cases the situation is much simpler, because a refinement of the ansatz is not required, while in the present case is absolutely necessary.\\

The structure of the paper is as follows.\\
In Section 2 we define the leading term and the correction terms of the solution and we show some asymptotic expansions; in Section 3 we provide an estimate of the error term and in Section 4 we study the invertibility of the linearized operator. In Section 5 we solve the auxiliary finite-dimensional problem and conclude the proof of the main theorem, whereas in Section 6 we discuss how to relax the assumption given in Definition \ref{ammis} on $V$. Finally, in the Appendix we compute the degree of the map \eqref{p}.\\

\section{Ansatz of the approximate solution}\label{s2}\

In this section we will give an \emph{ansatz} for our solutions, namely we precisely describe the profile of the solutions to \eqref{liou} we are looking for.\\

First of all, let us introduce the \emph{standard bubbles} for the Liouville equation, which are the main object in the study of the blow up analysis for \eqref{liou}. For $\rho,\tau>0$ and $x,\xi\in\R^2$, define
\begin{equation}\label{u}
U(x)=\log\frac{8\rho^2\tau^2}{\l(\rho^2\tau^2+|x-\xi|^2\r)^2};
\end{equation}
notice that they solve an equivalent problem to \eqref{liou} on the whole plane, namely
\begin{equation}\label{eqliou}
\l\{\begin{array}{ll}-\D U=e^U&\text{in }\R^2\\\int_{\R^2}e^U=8\pi<+\infty\end{array}\r..
\end{equation}
Actually, these are the only solutions of the previous equation with finite mass, as shown by \cite{chenli}.\\

The main order term in the ansatz will be the bubble $U$, for some $\tau=\tau_\rho$, satisfying $\frac1C\le\tau\le C$ for some $C>0$; it will be precisely defined later in the section. However, as we mentioned in the introduction, we also need some correction terms in order to have a sharper approximation. Our ansatz is the following,
$$\boxed{u(x)=W(x)+\phi(x)=\mrm PU(x)+\zeta_1(x)+\zeta_2(x)+\phi(x)},$$
where $\phi$ is a small order term and $\zeta_1,\zeta_2$ are suitable corrections which will be discussed in this section.\\
Correction terms will be related to the following \emph{two-variable functional}, depending on $x,\xi\in\O$:
\begin{equation}\label{e}
\mcal E(x,\xi)=e^{8\pi(H(x,\xi)-H(\xi,\xi))+\log\frac{V(x)}{V(\xi)}}-1.
\end{equation}
We will need the behavior of $\mcal E$ and its derivatives at the diagonal $x=\xi$ around the origin. The proof is based on elementary Taylor expansions and is therefore omitted.

\begin{lemma}\label{de}$ $\\
Let $\mcal E$, $\mcal F$ and $\mcal P$ be defined by \eqref{e}, \eqref{f} and \eqref{p}, respectively. Then, $\mcal E(x,\xi)$ verifies, as $\xi$ goes to $0$:
\begin{eqnarray*}
\mcal E(\xi,\xi)&=&0;\\
\pa_{x_i}\mcal E(\xi,\xi)&=&4\pi\pa_{\xi_i}\mcal F(\xi)=\frac1{2\pi}\pa_{\xi_i}\mcal P(\xi)+O\l(\l|\xi\r|^3\r)=O\l(|\xi|^2\r);\\
\pa^2_{x_ix_j}\mcal E(\xi,\xi)&=&-8\pi\pa^2_{x_iy_j}H(0,0)+O(|\xi|);\\
\pa^3_{x_ix_jx_k}\mcal E(\xi,\xi)&=&8\pi\pa^3_{x_ix_jx_k}H(0,0)+\pa^3_{x_ix_jx_k}\log V(0)+O(|\xi|).
\end{eqnarray*}
In particular, $\mcal E(x,\xi)=O\l(|\xi|^2|x-\xi|+|x-\xi|^2\r)$.
\end{lemma}\

\subsection{The local correction}\label{sec_pw}

Here we introduce the first correction $\zeta_1$ which basically plays a role in a small neighborhood of $\xi$. For this reason we call it the $local$ correction.\\

The second derivatives of $\mcal E$ appear in the first correction term, which solves an entire linear PDEs related to \eqref{eqliou}:
$$-\D\wh w(y)-\frac8{\l(1+|y|^2\r)^2}\wh w(y)=\frac4{\l(1+|y|^2\r)^2}\l\langle D^2_{xx}\mcal E(\xi,\xi),y,y\r\rangle.$$
Of course this equation will have many solutions and we will need some \qm{special} ones: we split it into three as
\begin{equation}\label{eqwh}
\wh w=\frac{\pa^2_{x_1x_1}\mcal E(\xi,\xi)+\pa^2_{x_2x_2}\mcal E(\xi,\xi)}2\wh w_1+\frac{\pa^2_{x_1x_1}\mcal E(\xi,\xi)-\pa^2_{x_2x_2}\mcal E(\xi,\xi)}2\wh w_2+\l(2\pa^2_{x_1x_2}\mcal E(\xi,\xi)\r)\wh w_3,
\end{equation}
with each $\wh w_i$ solving an equation involving only the radial term or one of the non-radial terms:
\begin{eqnarray}
\label{eqw1}-\D\wh w_1(y)-\frac8{\l(1+|y|^2\r)^2}\wh w_1(y)&=&\frac{4|y|^2}{\l(1+|y|^2\r)^2};\\
\nonumber-\D\wh w_2(y)-\frac8{\l(1+|y|^2\r)^2}\wh w_2(y)&=&\frac{4\l(y_1^2-y_2^2\r)}{\l(1+|y|^2\r)^2};\\
\nonumber-\D\wh w_3(y)-\frac8{\l(1+|y|^2\r)^2}\wh w_3(y)&=&\frac{4y_1y_2}{\l(1+|y|^2\r)^2}.
\end{eqnarray}
As for the latter two equations, we can choose two explicit bounded solutions given by:
\begin{equation}\label{w2w3}
\wh w_2(y):=\frac{y_1^2-y_2^2}{1+y_1^2+y_2^2};\q\q\q\q\q\q\wh w_3(y):=\frac{y_1y_2}{1+y_1^2+y_2^2}.
\end{equation}
On the other hand, solutions to \eqref{eqw1} have no explicit form, but radial solutions can be found using a variation of constants method for ODEs; they are not bounded but they may have a logarithmic control at infinity, as the following lemma shows.\\
The choice of this correction is similar to \cite{emp} (Lemma 2.1), but in our case the forcing term is not in $L^1$. For this reason, the asymptotic behavior is different and some estimates are more delicate. As we will see later, $\wh w_1$ determines the main order term in the behavior of $\wh w$, which will be crucial.

\begin{lemma}\label{w1}$ $\\
The following O.D.E.
\begin{equation}\label{ode}
-\wh w_1''(r)-\frac{\wh w_1'(r)}r-\frac8{\l(1+r^2\r)^2}\wh w_1(r)=\frac{4r^2}{\l(1+r^2\r)^2}\q\q\q r\in(0,+\infty)
\end{equation}
has a unique solution being smooth as $r$ goes to $0$ and additionally satisfying, as $r$ goes to $+\infty$,
\begin{equation}\label{w1inf}
\wh w_1(r)=-2\log^2r+4\log r+O\l(\frac{\log^2r}{r^2}\r).
\end{equation}
\end{lemma}

\begin{proof}$ $\\
Solutions to \eqref{ode} can be found using a standard variation of constants: since $w_0(r)=\frac{1-r^2}{1+r^2}$ solves the homogeneous equation
$$-w_0''(r)-\frac{w_0'(r)}r-\frac8{\l(1+r^2\r)^2}w_0(r)=0,$$
then all solutions are given, for some $C\in\R$, by
\begin{eqnarray*}
w(r)&=&-w_0(r)\l(\int_0^r\frac{\Phi(s)-\Phi(1)}{(1-s)^2}\mrm ds+\Phi(1)\frac r{r-1}+C\r),
\end{eqnarray*}
with
$$\Phi(s):=\frac{(1-s)^2}{sw_0(s)^2}\int_0^stw_0(t)\frac{4t^2}{\l(1+t^2\r)^2}\mrm dt,$$
extended by continuity in $s=1$; for details about the formula above, see for instance \cite{emp}, Lemma 2.1 and \cite{ggw}, Lemma 3.5.\\
As $s$ goes to $+\infty$, one has
$$\int_0^stw_0(t)\frac{4t^2}{\l(1+t^2\r)^2}\mrm dt=4\int_0^s\frac{t^3\l(1-t^2\r)}{\l(1+t^2\r)^3}\mrm dt=-4\log s+4+O\l(\frac1{s^2}\r),$$
therefore
\begin{eqnarray*}
w(r)&=&\l(1+O\l(\frac1{r^2}\r)\r)\l(\int_0^r\l(\frac1s+O\l(\frac1{s^3}\r)\r)\l(-4\log s+4+O\l(\frac1{s^2}\r)\r)+C\r)\\
&=&\l(1+O\l(\frac1{r^2}\r)\r)\l(\int_0^r\l(-4\frac{\log s}s+\frac4s+O\l(\frac{\log s}{s^3}\r)\r)\mrm ds+C\r)\\
&=&-2\log^2r+\log r+C_0+C+O\l(\frac{\log^2r}{r^2}\r).
\end{eqnarray*}
Therefore, $C=-C_0$ is the unique value for which the asymptotic behavior is as we wanted, hence we get the unique solution $\wh w_1=w$ with the desired properties.
\end{proof}\

We will consider the rescalement $\wh w\l(\frac{x-\xi}{\rho\tau}\r)$, which concentrates at $x=\xi$ as $\rho$ goes to $0$; for this reason, we will refer to $\wh w$ as the \emph{local} correction. Notice that $\wh w\l(\frac{x-\xi}{\rho\tau}\r)$ solves
$$-\D\Bigg(\wh w\l(\frac{x-\xi}{\rho\tau}\r)\Bigg)-e^{U(x)}\wh w\l(\frac{x-\xi}{\rho\tau}\r)=e^{U(x)}\frac1{2\rho^2\tau^2}\l\langle D^2_{xx}\mcal E(\xi,\xi),x-\xi,x-\xi\r\rangle;$$
moreover, in view of the asymptotic behavior of $\wh w_1$ and the boundedness of $\wh w_2,\wh w_3$, we also have $\wh w\l(\frac{x-\xi}{\rho\tau}\r)=O\l(\log^2\frac1\rho\r)$ on $\ol\O$.\\
Since we look for solutions to \eqref{liou} vanishing on $\pa\O$, we need to project also this correction on $H^1_0(\O)$, via the map $\mrm P:H^1(\O)\to H^1_0(\O)$ given by:
\begin{equation}\label{pr}
\l\{\begin{array}{ll}-\D(\mrm Pu)=-\D u&\text{in }\O\\\mrm Pu=0&\text{on }\O\end{array}\r..
\end{equation}

\begin{lemma}\label{wh}$ $\\
Let $\wh w(y)$ be defined by \eqref{eqwh}, \eqref{w2w3} and Lemma \ref{w1} and $\mrm P$ be defined by \eqref{pr}.\\
Then, as $x$ goes to $\xi$, it satisfies
\begin{eqnarray*}
\mrm P\wh w\l(\frac{x-\xi}{\rho\tau}\r)&=&\wh w\l(\frac{x-\xi}{\rho\tau}\r)-32\pi^2\log\frac1\rho\l(\pa^2_{x_1y_1}H(0,0)+\pa^2_{x_2y_2}H(0,0)\r)\langle\n_xH(0,0),x-\xi\rangle\\
&-&\l(\pa^2_{x_1x_1}\mcal E(\xi,\xi)+\pa^2_{x_2x_2}\mcal E(\xi,\xi)\r)\l(4\pi\l(1-\log\frac1{\rho\tau}\r)H(\xi,\xi)+\l(2-\log\frac1{\rho\tau}\r)\log\frac1{\rho\tau}\r)\\
&+&\l(\pa^2_{x_1x_1}\mcal E(\xi,\xi)+\pa^2_{x_2x_2}\mcal E(\xi,\xi)\r)L_1(\xi,\xi)-\frac{\pa^2_{x_1x_1}\mcal E(\xi,\xi)-\pa^2_{x_2x_2}\mcal E(\xi,\xi)}2L_2(\xi,\xi)\\
&-&2\pa^2_{x_1x_2}\mcal E(\xi,\xi)L_3(\xi,\xi)+O\l(\rho^2\log^2\frac1\rho+|x-\xi|+\l(\log\frac1\rho\r)|\xi||x-\xi|+\l(\log\frac1\rho\r)|x-\xi|^2\r),
\end{eqnarray*}
where $L_i(x,\xi)$ is respectively the solution to
\begin{eqnarray*}
&&\l\{\begin{array}{ll}-\D L_1(x,\xi)=0&x\in\O\\L_1(x,\xi)=\log^2|x-\xi|&x\in\pa\O\end{array}\r.\\
&&\l\{\begin{array}{ll}-\D L_2(x,\xi)=0&x\in\O\\L_2(x,\xi)=\frac{(x_1-\xi_1)^2-(x_2-\xi_2)^2}{|x-\xi|^2}&x\in\pa\O\end{array}\r.\\
&&\l\{\begin{array}{ll}-\D L_3(x,\xi)=0&x\in\O\\L_3(x,\xi)=\frac{(x_1-\xi_1)(x_2-\xi_2)}{|x-\xi|^2}&x\in\pa\O\end{array}\r.
\end{eqnarray*}
\end{lemma}

\begin{proof}$ $\\
From the asymptotic behavior \eqref{w1inf} we deduce that, for $x\in\partial\O$,
$$\wh w_1\l(\frac{x-\xi}{\rho\tau}\r)=-2\log^2|x-\xi|-4\l(\log\frac1{\rho\tau}\r)\log|x-\xi|-2\log^2\frac1{\rho\tau}+4\log|x-\xi|+4\log\frac1{\rho\tau}+O\l(\rho^2\log^2\frac1\rho\r);$$
therefore, from the maximum principle we get, uniformly in $\O$ as $\rho$ goes to $0$,
\begin{eqnarray*}
\mrm P\wh w_1\l(\frac{x-\xi}{\rho\tau}\r)&=&\wh w_1\l(\frac{x-\xi}{\rho\tau}\r)+2L_1(x,\xi)+\l(\log\frac1{\rho\tau}\r)8\pi H(x,\xi)+2\log^2\frac1{\rho\tau}-8\pi H(x,\xi)\\
&-&4\log\frac1{\rho\tau}+O\l(\rho^2\log\frac1\rho\r)\\
&=&\wh w_1\l(\frac{x-\xi}{\rho\tau}\r)+2L_1(\xi,\xi)+\l(\log\frac1{\rho\tau}\r)8\pi H(\xi,\xi)+8\pi\log\frac1{\rho\tau}\langle\n_xH(0,0),x-\xi\rangle\\
&+&2\log^2\frac1{\rho\tau}-8\pi H(\xi,\xi)-4\log\frac1{\rho\tau}\\
&+&O\l(\rho^2\log^2\frac1\rho+|x-\xi|+\l(\log\frac1\rho\r)|\xi||x-\xi|+\l(\log\frac1\rho\r)|x-\xi|^2\r).
\end{eqnarray*}
Similarly, from \eqref{w2w3} we get on $x\in\pa\O$,
\begin{eqnarray*}
\wh w_2\l(\frac{x-\xi}{\rho\tau}\r)&=&\frac{(x_1-\xi_1)^2-(x_2-\xi_2)^2}{|x-\xi|^2}+O\l(\rho^2\r)\\
\wh w_3\l(\frac{x-\xi}{\rho\tau}\r)&=&\frac{(x_1-\xi_1)(x_2-\xi_2)}{|x-\xi|^2}+O\l(\rho^2\r),
\end{eqnarray*}
therefore on $\O$, for $i=2,3$,
$$\mrm P\wh w_i\l(\frac{x-\xi}{\rho\tau}\r)=\wh w_i\l(\frac{x-\xi}{\rho\tau}\r)-L_i(x,\xi)+O\l(\rho^2\r)=\wh w_i\l(\frac{x-\xi}{\rho\tau}\r)-L_i(\xi,\xi)+O\l(\rho^2+|x-\xi|\r).$$
The conclusion follows by putting together the previous estimates and the asymptotic behavior of $\pa^2_{xx}\mcal E(\xi,\xi)$ from Lemma \ref{de}.
\end{proof}\

Finally we are in position to define our local correction term:
$$\boxed{\zeta_1(x)=\rho^2\tau^2\mrm P\wh w\l(\frac{x-\xi}{\rho\tau}\r)}.$$

\subsection{The global correction}\label{sec_w}

Let us now define the second correction term .\\
While $\wh w$ was introduced to compensate the effect of the second derivatives of $\mcal E$, $\zeta_2$ will deal with the other terms in the expansion of $\mcal E$. Anyway, unlike the former, it will be a solution of a PDE on the whole $\O$, rather than a concentrating rescaling of an entire solution.\\
Our global correction $\zeta_2$ is defined as
$$\boxed{\zeta_2(x)=\rho^2\tau^2\wt W(x)},$$
where $\wt W$ is the solution to the following Dirichlet problem:
\begin{equation}\label{wtilde}
\l\{\begin{array}{ll}-\D\wt W(x)=\frac8{|x-\xi|^4}\l(\mcal E(x,\xi)-\l\langle\n_x\mcal E(\xi,\xi),x-\xi\r\rangle-\frac12\l\langle D^2_{xx}\mcal E(\xi,\xi),x-\xi,x-\xi\r\rangle\r)&x\in\O\\\wt W(x)=0&x\in\pa\O\end{array}\r.
\end{equation}
Notice that, from the Taylor expansion of $\mcal E$, the right-hand side is bounded by constant times $\frac1{|x-\xi|}$, therefore it belongs to $L^p(\O)$ for any $p<2$ and from standard regularity $\wt W$ is H\"older continuous.\\
We point out that the presence of such a global term is a novelty in the construction of solutions to nonlinear problems, compared for instance to \cite{gro,emp,epw,gn}.\\
The asymptotic profile of $\wt W$ near $\xi$ is given by the following lemma.

\begin{lemma}\label{wt}$ $\\
Let $\wt W(x)$ be defined by \eqref{wtilde}.\\
Then, as $x$ goes to $\xi$, it satisfies
$$\wt W(x)=\frac12\langle\n\D\log V(0),x-\xi\rangle\log\frac1{|x-\xi|}+\wt W(\xi)+O\l(|\xi||x-\xi|\log\frac1{|x-\xi|}+|x-\xi|\r).$$
\end{lemma}

\begin{proof}$ $\\
We can write the right-hand side of \eqref{wtilde} as $f_1+f_2$, with $f_2\in L^\infty(\O)$ and
\begin{eqnarray*}
f_1(x)&:=&\frac{c_{111}(x_1-\xi_1)^3+3c_{112}(x_1-\xi_1)^2(x_2-\xi_2)+3c_{122}(x_1-\xi_1)(x_2-\xi_2)^2+c_{222}(x_2-\xi_2)^3}{|x-\xi|^4}\\
&=&\frac34(c_{111}+c_{122})\frac{x_1-\xi_1}{|x-\xi|^2}+\frac34(c_{112}+c_{222})\frac{x_2-\xi_2}{|x-\xi|^2}\\
&+&\frac{c_{111}-3c_{122}}4\frac{(x_1-\xi_1)^3-3(x_1-\xi_1)(x_2-\xi_2)^2}{|x-\xi|^4}\\
&+&\frac{3c_{112}-c_{222}}4\frac{3(x_1-\xi_1)^2(x_2-\xi_2)-(x_2-\xi_2)^3}{|x-\xi|^4},
\end{eqnarray*}
with $c_{ijk}:=\frac43\pa^3_{x_ix_jx_k}\mcal E(\xi,\xi)$.\\
Notice that a solution to $-\D\wt W_1=f_1$ is given by
\begin{eqnarray*}
\wt W_1(x)&:=&\frac38(c_{111}+c_{122})(x_1-\xi_1)\log\frac1{|x-\xi|}+\frac38(c_{112}+c_{222})(x_2-\xi_2)\log\frac1{|x-\xi|}\\
&+&\frac{c_{111}-3c_{122}}{32}\frac{(x_1-\xi_1)^3-3(x_1-\xi_1)(x_2-\xi_2)^2}{|x-\xi|^2}\\
&+&\frac{3c_{112}-c_{222}}{32}\frac{3(x_1-\xi_1)^2(x_2-\xi_2)-(x_2-\xi_2)^3}{|x-\xi|^2};
\end{eqnarray*}
since $\wt W_1(\xi)=0$, one has $\wt W=\wt W_1+\wt W_2$, with $\wt W_2$ solving $\l\{\begin{array}{ll}-\D\wt W_2=f_2&\text{in }\O\\\wt W_2=-\wt W_1&\text{on }\pa\O\end{array}\r.$, hence $\wt W_2\in C^1\l(\ol\O\r)$ and $\wt W_2(x)=\wt W_2(\xi)+O(|x-\xi|)=\wt W(\xi)+O(|x-\xi|)$. From this we get
$$\wt W(x)=\frac38(c_{111}+c_{122})(x_1-\xi_1)\log\frac1{|x-\xi|}+\frac38(c_{112}+c_{222})(x_2-\xi_2)\log\frac1{|x-\xi|}+\wt W(\xi)+O(|x-\xi|).$$
Finally, due to Lemma \ref{de} and the harmonicity of $\n_xH$, one gets
$$c_{111}+c_{122}=\frac43\pa_{x_1}\D\log V(0)+O(|\xi|),\q\q\q\q\q\q c_{112}+c_{222}=\frac43\pa_{x_2}\D\log V(0)+O(|\xi|),$$
which concludes the proof.
\end{proof}\

\subsection{The final ansatz}

We can finally give the ansatz for our problem: we look for solutions in the form:
\begin{equation}\label{w}
\boxed{u(x)=W(x)+\phi(x)=\mrm PU(x)+\rho^2\tau^2\l(\mrm P\wh w\l(\frac{x-\xi}{\rho\tau}\r)+\wt W(x)\r)+\phi(x)},
\end{equation}
with $U,\wh w,\wt W,\mrm P$ defined as before and $\phi$ to be found.\\

We conclude by giving the value of $\tau=\tau_\rho$.\\
$\tau$ is implicitly defined by the following equation, and it is easy to see that it is well-defined, continuously depends on $\rho$ and satisfies $\frac1C\le\tau\le C$.
\begin{eqnarray}
0&=&-\label{tau}\log\l(8\tau^2\r)+\log V(\xi)+8\pi H(\xi,\xi)+\rho^2\tau^2\Bigg(2I(\xi,\xi)\\
\nonumber&-&\l(\pa^2_{x_1x_1}\mcal E(\xi,\xi)+\pa^2_{x_2x_2}\mcal E(\xi,\xi)\r)\l(4\pi\l(1-\log\frac1{\rho\tau}\r)H(\xi,\xi)+\l(2-\log\frac1{\rho\tau}\r)\log\frac1{\rho\tau}\r)\\
\nonumber&+&\l(\pa^2_{x_1x_1}\mcal E(\xi,\xi)+\pa^2_{x_2x_2}\mcal E(\xi,\xi)\r)L_1(\xi,\xi)-\frac{\pa^2_{x_1x_1}\mcal E(\xi,\xi)-\pa^2_{x_2x_2}\mcal E(\xi,\xi)}2L_2(\xi,\xi)\\
\nonumber&-&2\pa^2_{x_1x_2}\mcal E(\xi,\xi)L_3(\xi,\xi)+\wt W(\xi)\Bigg).
\end{eqnarray}
Here, $L_1,L_2,L_3$ are defined as in Lemma \ref{wh} and $I$ is similarly defined as the solution to
\begin{equation}\label{i}
\l\{\begin{array}{ll}-\D I(x,\xi)=0&x\in\O\\I(x,\xi)=\frac1{|x-\xi|^2}&x\in\pa\O\end{array}\r.
\end{equation}
We point out that $\tau=\sqrt{\frac{V(0)}8}e^{4\pi H(0,0)}+O\l(\rho^2\log^2\frac1\rho\r)$. Anyway, we cannot just define $\tau:=\tau_0=\sqrt{\frac{V(0)}8}e^{4\pi H(0,0)}$ (which was done in \cite{egp}), since that more involved definition is essential to get sharper estimates in the following sections.\\

\section{Estimate of the error}\

This section is devoted to estimating the \emph{error term} $\mcal R$ defined by:
\begin{equation}\label{r}
\mcal R:=\D W+\rho^2Ve^W,
\end{equation}
where $W$ is defined in \eqref{w}.\\
Clearly, $\mcal R=0$ if and only if $W$ solves \eqref{liou}; the smaller $\mcal R$ is, the better is the approximation.\\
We will estimate the $L^p$ norm of $\mcal R$ for $p>1$ sufficiently close to $1$. $L^p$ norms for $p$ in similar ranges will be considered throughout the paper, hence one may suppose to fix some $p_0>1$ once and for all.\\

The following sharp estimate on $\mcal R$ also gives a clue on the optimal size of $\xi$, which in Section $5$ we will show to be $|\xi|=O\l(\rho\sqrt{\log\frac1\rho}\r)$.

\begin{proposition}\label{rlp}$ $\\
Let $\mcal R$ be defined by \eqref{r}.\\
Then, for $p>1$ suitably close to $1$,
$$\|\mcal R\|_p=O\l(\rho^{\frac2p-1}\l(\rho^2\log\frac1\rho+|\xi|^2\r)\r)$$
\end{proposition}

\begin{remark}$ $\\
We point out that, if $|\xi|=o(1)$ is suitably small, the correction terms in \eqref{w} considerably improves the estimate in Proposition \ref{rlp}. In fact, the ansatz $W=\mrm PU$ only gives $\|\mcal R\|_p=O\l(\rho^{\frac2p-1}\r)$ (see \cite{egp}, Lemma B.1).
\end{remark}\

To prove Proposition \ref{rlp} we will use an estimate on the difference between the prescribed term $W$ and the bubble $U$. We remark the presence of the term
\begin{equation}\label{theta}
\Th(x-\xi)=\rho^2\tau^2\l(\frac12\langle\n\D\log V(0),x-\xi\rangle\log\frac1{|x-\xi|}-32\pi^2\log\frac1\rho\l(\pa^2_{x_1y_1}H(0,0)+\pa^2_{x_2y_2}H(0,0)\r)\langle\n_xH(0,0),x-\xi\rangle\r),
\end{equation}
which will give rise to the term $\eta_0$ defined by \eqref{eta}

\begin{lemma}\label{wu}$ $\\
Let $W$ and $U$ be defined by \eqref{w} and \eqref{u} respectively. Then,
\begin{eqnarray*}
&&W(x)-U(x)\\
&=&8\pi(H(x,\xi)-H(\xi,\xi))-2\log\rho-\log V(\xi)+\rho^2\tau^2\wh w\l(\frac{x-\xi}{\rho\tau}\r)+\Th(x-\xi)\\
&+&O\l(\rho^4\log^2\frac1\rho+\rho^2|x-\xi|+\l(\rho^2\log\frac1\rho\r)|\xi||x-\xi|+\l(\rho^2\log\frac1\rho\r)|x-\xi|^2+\rho^2|\xi||x-\xi|\log\frac1{|x-\xi|}\r).
\end{eqnarray*}
\end{lemma}

\begin{proof}$ $\\
From Lemmas \ref{wh}, \ref{wt} and the definition \eqref{tau} of $\tau$ we get:
\begin{eqnarray*}
&&\rho^2\tau^2\l(\mrm P\wh w\l(\frac{x-\xi}{\rho\tau}\r)+\wt W(x)\r)\\
&=&\rho^2\tau^2\Bigg(-\l(\pa^2_{x_1x_1}\mcal E(\xi,\xi)+\pa^2_{x_2x_2}\mcal E(\xi,\xi)\r)\l(4\pi\l(1-\log\frac1{\rho\tau}\r)H(\xi,\xi)-\l(2-\log\frac1{\rho\tau}\r)\log\frac1{\rho\tau}\r)\\
&+&\l(\pa^2_{x_1x_1}\mcal E(\xi,\xi)+\pa^2_{x_2x_2}\mcal E(\xi,\xi)\r)L_1(\xi,\xi)+\frac{\pa^2_{x_1x_1}\mcal E(\xi,\xi)-\pa^2_{x_2x_2}\mcal E(\xi,\xi)}2L_2(\xi,\xi)\\
&-&2\pa^2_{x_1x_2}\mcal E(\xi,\xi)L_3(\xi,\xi)+\wt W(\xi)\Bigg)+\rho^2\tau^2\wh w\l(\frac{x-\xi}{\rho\tau}\r)+\Th(x-\xi)\\
&+&O\l(\rho^4\log^2\frac1\rho+\rho^2|x-\xi|+\l(\rho^2\log\frac1\rho\r)|\xi||x-\xi|+\l(\rho^2\log\frac1\rho\r)|x-\xi|^2+\rho^2|\xi||x-\xi|\log\frac1{|x-\xi|}\r)\\
&=&\log\l(8\tau^2\r)-\log V(\xi)-8\pi H(\xi,\xi)-2\rho^2\tau^2I(\xi,\xi)+\rho^2\tau^2\wh w\l(\frac{x-\xi}{\rho\tau}\r)+\Th(x-\xi)\\
&+&O\l(\rho^4\log^2\frac1\rho+\rho^2|x-\xi|+\l(\rho^2\log\frac1\rho\r)|\xi||x-\xi|+\l(\rho^2\log\frac1\rho\r)|x-\xi|^2+\rho^2|\xi||x-\xi|\log\frac1{|x-\xi|}\r).\\
\end{eqnarray*}
Moreover, the maximum principle and the definition \eqref{i} of $I(x,\xi)$ give:
\begin{eqnarray*}
\mrm PU(x)-U(x)&=&-\log\l(8\rho^2\tau^2\r)+8\pi H(x,\xi)+2\rho^2\tau^2I(x,\xi)+O\l(\rho^4\r)\\
&=&-\log\l(8\rho^2\tau^2\r)+8\pi H(x,\xi)+2\rho^2\tau^2I(\xi,\xi)+O\l(\rho^4+\rho^2|x-\xi|\r).
\end{eqnarray*}
By summing these two estimates and the definition \eqref{w} of $W$ the claim follows.
\end{proof}\

We will also need, here and later in the paper, some estimates on integrals of elementary functions. Since they are rather easy to prove and widely used in the study of problem \eqref{liou}, we skip the proof.

\begin{lemma}\label{lp}$ $\\
For $p>1$ suitably close to $1$ the following estimates hold true:
\begin{eqnarray*}
\l\|e^{U(x)}|x-\xi|^s\r\|_p&=&\l\{\begin{array}{ll}O\l(\rho^{\frac2p+s-2}\r)&\text{if }s\le2\\O\l(\rho^2\r)&\text{if }s>2\end{array}\r.;\\
\l\|e^{U(x)}|x-\xi|\log\frac1{|x-\xi|}\r\|_p&=&O\l(\rho^{\frac2p-1}\log\frac1\rho\r).
\end{eqnarray*}
\end{lemma}\

\begin{proof}[Proof of Proposition \ref{rlp}]$ $\\
Since $\wt W$ solves \eqref{wtilde}, we can write:
\begin{eqnarray*}
&&\rho^2\tau^2\D\wt W(x)\\
&=&-e^{U(x)}\l(1+\frac{2\rho^2\tau^2}{|x-\xi|^2}+\frac{\rho^4\tau^4}{|x-\xi|^4}\r)\l(\mcal E(x,\xi)-\l\langle\n_x\mcal E(\xi,\xi),x-\xi\r\rangle-\frac12\l\langle D^2_{xx}\mcal E(\xi,\xi),x-\xi,x-\xi\r\rangle\r)\\
&=&e^{U(x)}\Bigg(-\mcal E(x,\xi)+\l\langle\n_x\mcal E(\xi,\xi),x-\xi\r\rangle+\frac12\l\langle D^2_{xx}\mcal E(\xi,\xi),x-\xi,x-\xi\r\rangle\\
&+&O\l(\frac{\rho^4}{|x-\xi|}+\rho^2|x-\xi|\r)\Bigg).
\end{eqnarray*}
On the other hand, in view of Lemmas \ref{wu} and \ref{de}, we have:
\begin{eqnarray*}
&&\rho^2V(x)e^{W(x)-U(x)}\\
&=&(1+\mcal E(x,\xi))e^{\rho^2\tau^2\wh w\l(\frac{x-\xi}{\rho\tau}\r)+\Th(x-\xi)+O\l(\rho^4\log^2\frac1\rho+\rho^2|x-\xi|+\l(\rho^2\log\frac1\rho\r)|\xi||x-\xi|+\l(\rho^2\log\frac1\rho\r)|x-\xi|^2+\rho^2|\xi||x-\xi|\log\frac1{|x-\xi|}\r)}\\
&=&1+\mcal E(x,\xi)+\rho^2\tau^2\wh w\l(\frac{x-\xi}{\rho\tau}\r)+\Th(x-\xi)\\
&+&O\l(|\mcal E(x,\xi)|\l|\Th(x-\xi)+\rho^2\tau^2\wh w\l(\frac{x-\xi}{\rho\tau}\r)\r|+\l|\Th(x-\xi)+\rho^2\tau^2\wh w\l(\frac{x-\xi}{\rho\tau}\r)\r|^2\r)\\
&+&O\l(\rho^4\log^2\frac1\rho+\rho^2|x-\xi|+\l(\rho^2\log\frac1\rho\r)|\xi||x-\xi|+\l(\rho^2\log\frac1\rho\r)|x-\xi|^2+\rho^2|\xi||x-\xi|\log\frac1{|x-\xi|}\r)\\
&=&1+\mcal E(x,\xi)+\rho^2\tau^2\wh w\l(\frac{x-\xi}{\rho\tau}\r)+\Th(x-\xi)\\
&+&O\l(\rho^4\log^4\frac1\rho+\rho^2|x-\xi|+\l(\rho^2\log^2\frac1\rho\r)|\xi||x-\xi|+\l(\rho^2\log\frac1\rho\r)|x-\xi|^2+\rho^2|\xi||x-\xi|\log\frac1{|x-\xi|}\r),\\
\end{eqnarray*}
where we used that $\wh w\l(\frac{x-\xi}{\rho\tau}\r)=O\l(\log^2\frac1\rho\r)$ on $\ol\O$.\\
Therefore, using the previous estimates and the expansion of $\n_x\mcal E(\xi,\xi)$ given by Lemma \ref{de}, we get:
\begin{eqnarray}
\nonumber\mcal R(x)&=&\D\mrm PU(x)+\rho^2\tau^2\D\mrm P\wh w\l(\frac{x-\xi}{\rho\tau}\r)+\rho^2\tau^2\D\wt W(x)+\rho^2V(x)e^{W(x)}\\
\nonumber&=&e^{U(x)}\Bigg(-1-\rho^2\tau^2\wh w\l(\frac{x-\xi}{\rho\tau}\r)-\frac12\l\langle D^2_{xx}\mcal E(\xi,\xi),x-\xi,x-\xi\r\rangle\\
\nonumber&-&\mcal E(x,\xi)+\l\langle\n_x\mcal E(\xi,\xi),x-\xi\r\rangle+\frac12\l\langle D^2_{xx}\mcal E(\xi,\xi),x-\xi,x-\xi\r\rangle+\rho^2V(x)e^{W(x)-U(x)}\\
\nonumber&+&O\l(\frac{\rho^4}{|x-\xi|}+\rho^2|x-\xi|\r)\Bigg)\\
\nonumber&=&e^{U(x)}\Bigg(\frac1{2\pi}\langle\n\mcal P(\xi),x-\xi\rangle+\Th(x-\xi)\\
\label{rx}&+&O\l(\frac{\rho^4}{|x-\xi|}+\rho^2|x-\xi|+\l(\rho^2\log^2\frac1\rho\r)|\xi||x-\xi|+\l(\rho^2\log^2\frac1\rho\r)|x-\xi|^2+\rho^2|\xi||x-\xi|\log\frac1{|x-\xi|}\r)\Bigg).
\end{eqnarray}
Since $|\n\mcal P(\xi)|=O\l(|\xi|^2\r)$ and $|\Th(x-\xi)|=O\l(\rho^2|x-\xi|\log\frac1{|x-\xi|}+\l(\rho^2\log\frac1\rho\r)|x-\xi|\r)$, Lemma \ref{lp} gives the desired estimates.
\end{proof}\

\section{Linear theory}\

In this section we will apply a fixed point theorem in some suitable spaces to find $\phi$.\\
To this purpose, one sees that $u=W+\phi$ solves \eqref{liou} if and only if $\phi$ solves
\begin{equation}\label{eqphi}
\mcal R+\mcal L\phi+\mcal N(\phi)=0,
\end{equation}
where $\mcal R$ is the error defined by \eqref{r}, $\mcal L$ is the linearized operator at $\phi=0$, given by
\begin{equation}\label{l}
\mcal L\phi:=\D\phi+\rho^2Ve^W\phi,
\end{equation}
and $\mcal N$ is the \emph{nonlinear term}:
\begin{equation}\label{n}
\mcal N(\phi):=\rho^2Ve^W\l(e^\phi-1-\phi\r).
\end{equation}\

In order to solve \eqref{eqphi} we investigate the invertibility of the linearized operator $\mcal L$.\\
The operator $\mcal L$ will not be invertible on the whole space $H^1_0(\O)$, but it will be on a finite-codimensional space. In particular, if we define $\psi_1,\psi_2\in H^1_0(\O)$ as
\begin{equation}\label{psi}
\psi_i(x):=\frac{x_i-\xi_i}{\rho^2\tau^2+|x-\xi|^2},
\end{equation}
we can invert $\mcal L$ on the orthogonal complement of $\mbf K=\mrm{span}\{\mrm P\psi_1,\mrm P\psi_2\}$, namely
\begin{equation}\label{kort}
\mbf K^\perp=\l\{\phi\in H^1_0(\O):\,\langle\phi,\mrm P\psi_1\rangle_{H^1_0(\O)}=\langle\phi,\mrm P\psi_2\rangle_{H^1_0(\O)}=0\r\}.
\end{equation}
Notice that the $\psi_i$'s solve the linear problem $-\D\mrm P\psi_i=-\D\psi_i=e^U\psi_i$, with $U$ being the standard bubble \eqref{u}.
The estimates on the inverse operator $\mcal L^{-1}$ are not uniform in $\rho$, as its norm diverges logarithmically as $\rho$ goes to $0$. However, this is not an issue since most estimates throughout the paper, such as Proposition \ref{rlp}, converge polynomially in $\rho$.\\

The results in this section are obtained arguing very similarly to \cite{egp}, since the main term $\mrm PU$ in the ansatz \eqref{w} is the same as in \cite{egp} and the correction terms are negligible. Therefore proofs will be sketchy or skipped.\\
The following Lemma, concerning invertibility of $\mcal L$, is analogous to Proposition 3.1 in \cite{egp}.

\begin{lemma}\label{lin}$ $\\
Let $\mbf K^\perp$ and $\mcal L$ be defined respectively by \eqref{kort} and \eqref{l} and $\xi\in\O$, $f\in L^p(\O)$ be given with $p>1$.\\
Then, there exists $\rho_0>0$ such that for any $\rho\in(0,\rho_0)$ there is a unique solution $(\phi,c_1,c_2)\in\mbf K^\perp\x\R\x\R$ to
\begin{equation}\label{lphi}
\mcal L\phi=f+e^U(c_1\psi_1+c_2\psi_2).
\end{equation}
Moreover, there exists $C>0$, not depending on $\rho$ nor on $\xi$ provided $\xi$ does not approach $\pa\O$, such that
$$\|\phi\|\le C\log\frac1\rho\|f\|_p.$$
\end{lemma}

\begin{proof}[Sketch of the proof]$ $\\
Following \cite{egp}, we argue by contradiction, assuming there is a family of solutions $\phi\in\mbf K^\perp$ to \eqref{lphi} satisfying $\|\phi\|=1$ and $\|f\|_p=o\l(\frac1{\log\frac1\rho}\r)$.\\
By multiplying each side of \eqref{lphi} with each $\mrm P\psi_i$ we deduce $|c_i|=O\l(\rho^\frac2p\r)$. Then, by testing again suitable functions one gets $\langle\phi,\mrm P\psi_0\rangle_{H^1_0(\O)}=o(1)$, where $\psi_0=\frac{\rho^2\tau^2-|x-\xi|^2}{\rho^2\tau^2+|x-\xi|^2}$ is another solution to $-\D\psi_0=e^U\psi_0$.\\
Finally, one considers a rescaling $\wt\phi(y)=\phi(\rho\tau y+\xi)$, which is uniformly bounded with respect to the norm defined by
$$\l\|\wt\phi\r\|^2:=\int_{\R^2}\l(\l|\n\wt\phi(y)\r|^2+\frac{\l|\wt\phi(y)\r|^2}{\l(1+|y|^2\r)^2}\r)\mrm dy.$$
The weak limit $\wt\phi_0$ must solve $-\D\wt\phi_0(y)=\frac8{\l(1+|y|^2\r)^2}\wt\phi_0(y)$; however, since $\phi$ is orthogonal to $\mrm P\psi_1,\mrm P\psi_2$ and almost orthogonal to $\mrm P\psi_0$, the limit is $0$. This leads to a contradiction.
\end{proof}
We have the following estimate on the nonlinear term $\mcal N$, a sort of counterpart of Proposition \ref{rlp} on $\mcal R$ and Lemma \ref{lin} on $\mcal L$.\\
Such a result can be deduced by elementary inequalities and the estimates in Lemma \ref{lp}. The proof is roughly the same as Lemma B.2 in \cite{egp}, therefore it will be skipped.

\begin{lemma}\label{nonlin}$ $\\
Let $\mcal N$ be defined by \eqref{n}.\\
Then, for $p,q>1$, there exists $C>0$, not depending on $\rho$ nor on $\xi$ provided $\xi$ does not approach $\pa\O$, such that for any $\phi,\phi'\in\mbf K^\perp$
$$\|\mcal N(\phi)-\mcal N(\phi')\|_p\le C\rho^{{\frac2{pq}-2}}\|\phi-\phi'\|(\|\phi\|+\|\phi'\|)e^{C\l(\|\phi\|^2+\l\|\phi'\r\|^2\r)}.$$
In particular, if $\phi'=0$, one has
$$\|\mcal N(\phi)\|_p\le C\rho^{{\frac2{pq}-2}}\|\phi\|^2e^{C\|\phi\|^2}$$
\end{lemma}\

We are now in position to apply a fixed point theory on a suitably small ball of $\mbf K^\perp$. As we are not on the whole space $H^1_0(\O)$, we will solve equation \eqref{eqphi} only on the space $\mbf K^\perp$; in other words, on the right-hand side we will find a possibly non-zero element of $\mbf K$, depending on $\xi$. This issue will be addressed in the next section.

\begin{lemma}\label{contr}$ $\\
Let $\xi\in\O$ be given.\\
Then, there exists $\rho_0>0$ such that, if $p$ is suitably close to $1$, for any $\rho\in(0,\rho_0)$ there is a unique solution $(\phi_\xi,c_{1\xi},c_{2\xi})\in\mbf K^\perp\x\R\x\R$ to
\begin{equation}\label{eqphixi}
\mcal R+\mcal L\phi_\xi+\mcal N(\phi_\xi)=e^U(c_{1\xi}\psi_1+c_{2\xi}\psi_2).
\end{equation}
Moreover, there exists $C>0$, not depending on $\rho$ nor on $\xi$ provided $\xi$ does not approach $\pa\O$, such that
\begin{equation}\label{phinorma}
\|\phi_\xi\|\le C\rho^{\frac2p-1}\log\frac1\rho\l(\rho^2\log\frac1\rho+|\xi|^2\r).
\end{equation}
\end{lemma}

\begin{proof}[Sketch of the proof]$ $\\
From Lemma \ref{lin} we can define an invertible operator
$$\wt{\mcal L}:=\Pi\c(-\D)^{-1}\c\mcal L:\mbf K^\perp\to\mbf K^\perp,$$
where $\Pi:H^1_0(\O)\to\mbf K^\perp$ is the standard projection in Hilbert spaces and $(-\D)^{-1}:L^p(\O)\to H^1_0(\O)$ is the inverse of the Laplacian with Dirichlet conditions. From Lemma \ref{lin} and Sobolev embeddings we also deduce $\l\|\wt{\mcal L}^{-1}\phi\r\|\le C\log\frac1\rho\|\phi\|$.\\
In view of this, any solution of \eqref{eqphixi} is a fixed point of the map $\mcal T:\mbf K^\perp\to\mbf K^\perp$ defined by
$$\mcal T:\phi\map\wt{\mcal L}^{-1}\c\Pi\c(-\D)^{-1}\l(-\mcal R-\mcal N(\phi)\r).$$
Proposition \ref{rlp} and Lemma \ref{nonlin} give the following estimates:
\begin{eqnarray*}
\|\mcal T(\phi)\|&\le&C\log\frac1\rho\l(\|\mcal R\|_p+\|\mcal N(\phi)\|_p\r)\le C\log\frac1\rho\l(\rho^{\frac2p-1}\l(\rho^2\log\frac1\rho+|\xi|^2\r)+\rho^{\frac2{pq}-2}\|\phi\|^2e^{C\|\phi\|^2}\r)\\
\|\mcal T(\phi)-\mcal T(\phi')\|&\le&C\log\frac1\rho\|\mcal N(\phi)-\mcal N(\phi')\|\le C\log\frac1\rho\rho^{\frac2{pq}-2}\|\phi-\phi'\|(\|\phi\|+\|\phi'\|)e^{C\l(\|\phi\|^2+\l\|\phi'\r\|^2\r)}.
\end{eqnarray*}
If we take $R$ large enough, $\rho$ small enough and $q<\frac2{3p-2}$, then
$$\|\phi\|\le R\rho^{\frac2p-1}\log\frac1\rho\l(\rho^2\log\frac1\rho+|\xi|^2\r)\q\q\q\To\q\q\q\|\mcal T(\phi)\|\le R\rho^{\frac2p-1}\log\frac1\rho\l(\rho^2\log\frac1\rho+|\xi|^2\r)$$
and moreover $\sup_{\phi\ne\phi'}\frac{\|\mcal T(\phi)-\mcal T(\phi')\|}{\|\phi-\phi'\|}<1$. Therefore $\mcal T$ is a contraction on a suitable ball in $\mbf K^\perp$ and has a fixed point $\phi$ which verifies \eqref{eqphixi} and \eqref{phinorma}.
\end{proof}\

\section{Finite-dimensional problem and conclusion}\

We will now discuss the proper choice of $\xi=\xi_\rho$ in order to conclude the proof of Theorem \ref{molt}.\\

In the previous section we showed that, for any $\xi$, one can find $\phi$ so that $\mcal R+\mcal L\phi+\mcal N(\phi)$ is a linear combination of $e^U\psi_1$ and $e^U\psi_2$. Therefore, to get a solution to \eqref{eqphi}, hence to \eqref{liou}, we are left to show that, for some $\xi$, $\mcal R+\mcal L\phi+\mcal N(\phi)$ is somehow orthogonal to each $\mrm P\psi_i$. In particular, since we are interesting in multiplicity of solution, we want to find at least two $\xi^1,\xi^2$ satisfying such an orthogonality condition.\\
The following proposition gives the leading term of the integral against $\mrm P\psi_i$. 
\begin{proposition}\label{eqxi}$ $\\
Let $\xi$ satisfy $|\xi|=O\l(\rho\sqrt{\log\frac1\rho}\r)$ and $\phi=\phi_\xi$ be as in Lemma \ref{contr}.\\
Then,
\begin{equation}\label{eq}
\int_\O(\mcal R+\mcal L\phi+\mcal N(\phi))\mrm P\psi_i=\pa_{\xi_i}\mcal P(\xi)-\eta_i+O\l(\rho^2\r)\end{equation}
where $\eta=\l(\rho^2\log\frac1\rho\r)\eta_0$ and $\eta_0\in\R^2$ is defined in \eqref{eta}.
\end{proposition}\

We first prove that $\phi$ plays no role in the orthogonality condition, nor the projection $\mrm P$ does.\\
We stress that the choice of a refined ansatz is essential in order that $\phi$ is negligible in these computations, which in turn is essential to get explicit conditions on $\xi$.

\begin{lemma}\label{nophi}$ $\\
Let $\phi=\phi_\xi$ be as in Lemma \ref{contr}.\\
Then,
$$\int_\O(\mcal R+\mcal L\phi+\mcal N(\phi))\mrm P\psi_i=\int_\O\mcal R\psi_i+O\l(\rho^{\frac7p-6}\l(\rho^2\log\frac1\rho+|\xi|^2\r)\r).$$
\end{lemma}

\begin{proof}$ $\\
First we observe that by the maximum principle $\mrm P\psi_i-\psi_i$ is uniformly bounded in $\ol\O$.\\
From this, we also get
$$\|\mrm P\psi_i\|^2=\int_\O\mrm P\psi_i(-\D\mrm P\psi_i)=\int_\O(\psi_i+O(1))e^U\psi_i=\int_\O e^U\psi_i^2+O(1)\int_\O e^U|\psi_i|=O\l(\frac1{\rho^2}\r).$$
Therefore, from the previous estimates and Proposition \ref{rlp},
\begin{equation}\label{rpsi}
\int_\O\mcal R\mrm P\psi_i-\int_\O\mcal R\psi_i=O\l(\|\mcal R\|_p\|\mrm P\psi_i-\psi_i\|_\frac p{p-1}\r)=O\l(\rho^{\frac2p-1}\l(\rho^2\log\frac1\rho+|\xi|^2\r)\r).
\end{equation}
As for the linear term, we integrate by parts and write
\begin{eqnarray}
\nonumber\int_\O(\mcal L\phi)\mrm P\psi_i&=&\int_\O(-\phi e^U\psi_i+\rho^2Ve^W\phi\mrm P\psi_i)\\
\nonumber&=&\int_\O\phi(e^U(\mrm P\psi_i-\psi_i)+\l(\rho^2Ve^W-e^U\r)\mrm P\psi_i)\\
\nonumber&=&O\l(\|\phi\|\l(\l\|e^U\r\|_p\|\mrm P\psi_i-\psi_i\|_\infty+\l\|\rho^2Ve^W-e^U\r\|_p|\mrm P\psi_i\|\r)\r)\\
\label{lpsi}&=&O\l(\rho^{\frac4p-3}\log\frac1\rho\l(\rho^2\log\frac1\rho+|\xi|^2\r)\r),
\end{eqnarray}
where we used the estimates $\l\|e^U\r\|_p=O\l(\rho^{\frac2p-2}\r)$, and then
\begin{eqnarray}
\nonumber\l\|\rho^2Ve^W-e^U\r\|_p&=&\l\|e^U\l(e^{W-U+2\log\rho+\log V}-1\r)\r\|_p\\
\nonumber&=&\l\|e^U\l(\mcal E(x,\xi)+O\l(\rho^2\log^2\frac1\rho\r)\r)\r\|_p\\
\nonumber&=&\l\|e^UO\l(\rho^2\log\frac1\rho+|x-\xi|\r)\r\|_p\\
\label{wulp}&=&O\l(\rho^{\frac2p-1}\r),
\end{eqnarray}
in view of Lemmas \ref{wu} and \ref{lp}.\\
Finally, from Lemma \ref{nonlin} we get
\begin{equation}\label{npsi}
\int_\O\mcal N(\phi)\mrm P\psi_i=O(\|\mcal N\|_p\|\mrm P\psi\|)=O\l(\rho^{\frac2{pq}-3}\|\phi\|^2\r)=O\l(\rho^{\frac2{pq}+\frac4p-5}\log^2\frac1\rho\l(\rho^2\log\frac1\rho+|\xi|^2\r)^2\r).
\end{equation}
If $q$ is chosen close enough to $1$, then we conclude by summing the estimates \eqref{rpsi}, \eqref{lpsi}, \eqref{npsi}.
\end{proof}\

We will also need some integral computations involving $\psi_i$, in a similar spirit to Lemma \ref{lp}. The proof of the following Lemma is an easy computation and will be skipped.

\begin{lemma}\label{intpsi}$ $\\
Let $\psi_i$ be defined by \eqref{psi}.\\
Then, for $p>1$ suitably close to $1$ the following estimates hold true:
\begin{eqnarray*}
\int_\O e^{U(x)}|x-\xi|^s|\psi_j(x)|\mrm dx&=&\l\{\begin{array}{ll}O\l(\rho^{s-1}\r)&\text{if }s<3\\O\l(\rho^2\log\frac1{\rho}\r)&\text{if }s=3\\O\l(\rho^2\r)&\text{if }s>3\end{array}\r.;\\
\int_\O e^{U(x)}(x_i-\xi_i)\psi_j(x)\mrm dx&=&2\pi\d_{ij}+O\l(\rho^2\r);\\
\int_\O e^{U(x)}(x_i-\xi_i)\log\frac1{|x-\xi|}\psi_j(x)\mrm dx&=&2\pi\d_{ij}\log\frac1\rho+O(1);\\
\end{eqnarray*}
\end{lemma}\

\begin{proof}[Proof of Proposition \ref{eqxi}]$ $\\
From the estimate \eqref{rx} we get
$$\int_\O\mcal R\psi_i=\int_\O e^{U(x)}\l(\frac1{2\pi}\langle\n\mcal P(\xi),x-\xi\rangle+\Th(x-\xi)+O\l(\frac{\rho^4}{|x-\xi|}+\rho|x-\xi|^2+\l(\rho^2\log^2\frac1\rho\r)|\xi|\sqrt{|x-\xi|}\r)\r)\psi_i(x)\mrm dx.$$
By Lemma \ref{intpsi}, the definitions \eqref{theta}, \eqref{eta} respectively of $\Th,\eta$ and the estimate $\tau^2=\frac{V(0)}8e^{8\pi H(0,0)}+O\l(\rho^2\log^2\frac1\rho\r)$ we deduce:
\begin{eqnarray*}
\int_\O e^{U(x)}\langle\n\mcal P(\xi),x-\xi\rangle\psi_i(x)\mrm dx&=&2\pi\pa_{\xi_i}\mcal P(\xi)+O\l(\rho^2\r);\\
\int_\O e^{U(x)}\Th(x-\xi)\psi_i(x)\mrm dx&=&\eta_i+O\l(\rho^2\r);\\
\int_\O e^{U(x)}O\l(\frac{\rho^4}{|x-\xi|}+\rho^2|x-\xi|+\l(\rho^2\log^2\frac1\rho\r)|\xi|\sqrt{|x-\xi|}\r)|\psi_i(x)|\mrm dx&=&O\l(\rho^2+\l(\rho^\frac32\log^2\frac1\rho\r)|\xi|\r).
\end{eqnarray*}
By the assumption on $|\xi|$, then the error in the last term and in Lemma \ref{nophi} is also $O\l(\rho^2\r)$, therefore:
$$\int_\O(\mcal R+\mcal L\phi+\mcal N(\phi))\mrm P\psi_i=\int_\O\mcal R\psi_i+O\l(\rho^2\r)=\pa_{\xi_i}\mcal P(\xi)-\eta_i+O\l(\rho^2\r).$$
\end{proof}\

We are finally in condition to prove the main result of the paper.

\begin{proof}[Proof of Theorem \ref{molt}]$ $\\
For sake of simplicity we only consider the case $K=2$, namely the equation $\n\mcal P(\xi)-\eta=0$ has two distinct stable solutions.\\
Therefore, since this is the leading term in \eqref{eq}, there will be two stable $\xi^1\ne\xi^2$ such that \eqref{eq} vanishes. Thanks to Lemma \ref{contr}, $\phi_{\xi^1},\phi_{\xi^2}$ also solve \eqref{eqphixi}, hence
\begin{eqnarray*}
0&=&\int_\O\l(\mcal R+\mcal L\phi_{\xi^i}+\mcal N\l(\phi_{\xi^i}\r)\r)\l(c_{1\xi^i}\mrm P\psi_1+c_{2\xi^i}\mrm P\psi_2\r)\\
&=&\int_\O e^U\l(c_{1\xi^i}\psi_1+c_{2\xi^i}\psi_2\r)\l(c_{1\xi^i}\mrm P\psi_1+c_{2\xi^i}\mrm P\psi_2\r)\\
&=&\int_\O\l|\n\l(c_{1\xi^i}\mrm P\psi_1+c_{2\xi^i}\mrm P\psi_2\r)\r|^2;\\
\end{eqnarray*}
since $\int_\O\n\mrm P\psi_1\cdot\n\mrm P\psi_2=O\l(\int_\O|\n\mrm P\psi_i|^2\r)$ for $i=1,2$ (see \cite{egp}, Lemma A.4), then $c_{1\xi^i}=c_{2\xi^i}=0$, namely the $\phi_{\xi^i}$'s solve \eqref{eqphi} and $u^i=W+\phi_{\xi^i}$ are solutions to \eqref{liou}.\\
Let us now show that each $u^i$ blows up at $0$. To this purpose, we need some estimates in $L^\infty$: from \eqref{wulp}, for $q>p>1$ we have:
\begin{eqnarray*}
\|-\D\phi\|_p&=&\l\|\rho^2Ve^W\phi+\mcal R+\mcal N(\phi)\r\|_p\\
&\le&\l(\l\|\rho^2Ve^W-e^U\r\|_q+\l\|e^U\r\|_q\r)\|\phi\|+\|\mcal R\|_p+\|\mcal N(\phi)\|_p\\
&\le&C\rho^{\frac2q+\frac2p-1}\log^2\frac1\rho;
\end{eqnarray*}
then $\|\phi\|_\infty\le C\|-\D\phi\|_p=o(1)$. Similarly, by construction,
$$\l\|\rho^2\tau^2\l(\mrm P\wh w\l(\frac{x-\xi}{\rho\tau}\r)+\wt W(x)\r)\r\|_\infty=o(1),$$
therefore all terms in $W$ but the main one vanish in $L^\infty$. Concerning the latter, we use the maximum principle to get
\begin{equation}\label{ui}
u^i(x)=\mrm PU(x)+o(1)=U(x)+\log\l(8\rho^2\tau^2\r)+8\pi H(x,\xi)+o(1)=\log\frac1{\l(\rho^2+|x-\xi^i|^2\r)^2}+O(1),
\end{equation}
which implies blow up in the sense of Definition \ref{blowup}.\\
Finally, after rescaling $\xi^i=\rho\sqrt{\log\frac1\rho}\xi^i_0$, one has $\n\mcal P\l(\xi_0^i\r)=\eta_0$ and $\frac1C\le\l|\xi^1_0-\xi^2_0\r|\le C$, therefore \eqref{ui} gives
\begin{eqnarray*}
\l|u^1\l(\xi^1_0\r)-u^2\l(\xi^1_0\r)\r|&=&\l|\log\frac1{\rho^4}-\log\frac1{\l(\rho^2+|\xi^2_0-\xi^1_0|^2\rho^2\log\frac1\rho\r)^2}\r|+O(1)\\
&=&2\log\l(\log\frac1\rho\r)+O(1)\\
&\us{\rho\to0}\to&+\infty,
\end{eqnarray*}
which proves $u^1\not\equiv u^2$.
\end{proof}\

\section{Higher-order degeneracy}\label{hod}

In the last section we discuss some extensions to Theorem \ref{molt} to some more general case.\\

Throughout all the paper we have assumed some degeneracy conditions on the first and second derivatives of $\mcal F$ at $0$ and non-degeneracy of its third derivatives, according to Definition \ref{ammis}.\\
This can be generalized by assuming to be zero also the third derivatives and all other derivatives up to order $N$, and then non-degeneracy on derivatives of order $N+1$. As before, all these conditions on the derivatives of $\mcal F$ in zero can be obtained by a proper choice of $V$. Precisely, we will make the following assumption on $V$.

\begin{definition}\label{ammis2}$ $\\
Let $V\in C^\infty\l(\ol\O\r)$ be a positive potential.\\
We say that $V$ is \emph{admissible of order $N$} if there exists $N\in\N$ such that the functional $\mcal F$ defined by \eqref{f} satisfies the following properties:
\begin{itemize}
\item $\pa^n_{\xi_{i_1}\dots\xi_{i_n}}\mcal F(0)=0$ for all $n=1,\dots,N$, $i_1,\dots,i_n=1,2$;
\item The map $\mcal P(\xi)$ defined by
\begin{equation}\label{p2}
\mcal P(\xi)=\frac{4\pi^2}{N+1}\l\langle D^{N+1}\mcal F(0),\xi,\dots,\xi\r\rangle=\frac{4\pi^2}{N+1}\sum_{i_1,\dots,i_N=1}^2\l(\pa^{N+1}_{\xi_{i_1}\dots\xi_{i_{N+1}}}\mcal F(0)\r)\xi_{i_1}\dots\xi_{i_N}
\end{equation}
has $\xi=0$ as its only critical point.
\end{itemize}
\end{definition}\

Most of the result obtained in the first part of this paper are still valid under assuming $V$ to be admissible of order $N$. In fact, in Sections 2,3,4 the non-degeneracy of third derivatives of $\mcal F$ are never used; moreover, Proposition \ref{eqxi}, the main result in Section $5$, can be generalized so that the new condition on $\xi$ is of the kind $\n\mcal P(\xi)=\eta$, with $\mcal P$ now being defined by \eqref{p2}.\\
The main difference between the two cases is that the vector $\eta_0$ defined in \eqref{eta} may vanish. In fact, the new assumption on $D^3\mcal F(0)$ gives no more freedom in the choice of $\eta_0$, which only depends on the derivatives of $H$ in $0$.\\
In particular, if $\O$ is simply connected, then due to the properties of $H$ one always gets $\eta_0\eq0$, therefore the only optimal $\xi$ is $0$ and multiplicity of blowing-up solutions fails. On the other hand, if $\O$ is \emph{not} simply connected, then $\eta_0$ is not zero, up to possibly translate the domain, hence construction of multiple solutions still works; furthermore, by suitably choosing $N$ and $\mcal P$, one can get as many solutions as desired.\\
This different phenomena affecting simply and multiply connected domain is somehow surprising, although consistent with well-known obstructions in the existence of solutions to \eqref{liou} in simply connected domains (see also \cite{bat}).\\
The picture is described by the following lemma:

\begin{lemma}\label{semplcon}$ $\\
Assume $V$ is an admissible potential of order $N\ge3$. Then, the vector $\eta_0$ defined by \eqref{eta} has the form
\begin{equation}\label{eta2}
\eta_0=\l(64\pi^3\l(\pa^2_{x_1y_1}H(0,0)+\pa^2_{x_2y_2}H(0,0)\r)\n_xH(0,0)+16\pi^2\n_x\l(\pa^2_{x_1y_1}H+\pa^2_{x_2y_2}H\r)(0,0)\r)\frac{V(0)}8e^{8\pi H(0,0)}.
\end{equation}
Moreover, if $\O$ is simply connected then $\eta_0=0$.\\
If $\O$ is multiply connected, then for some $\xi\in\O$ one has $\eta_0\ne0$ in the domain $\O+\xi$.
\end{lemma}\

\begin{proof}$ $\\
Under these assumptions, all the third order derivatives of $\mcal F$ vanish in $\xi=0$, therefore due to the symmetry of $H$:
\begin{eqnarray*}
0&=&\pa^3_{\xi_1\xi_1\xi_1}\mcal F(0)=2\pa^3_{x_1x_1x_1}H(0,0)+6\pa^3_{x_1x_1y_1}H(0,0)+\frac1{4\pi}\pa^3_{x_1x_1x_1}\log V(0);\\
0&=&\pa^3_{\xi_1\xi_1\xi_2}\mcal F(0)=2\pa^3_{x_1x_1x_2}H(0,0)+2\pa^3_{x_1x_1y_2}H(0,0)+4\pa^3_{x_1x_2y_1}H(0,0)+\frac1{4\pi}\pa^3_{x_1x_1x_2}\log V(0);\\
0&=&\pa^3_{\xi_1\xi_2\xi_2}\mcal F(0)=2\pa^3_{x_1x_2x_2}H(0,0)+2\pa^3_{x_2x_2y_1}H(0,0)+4\pa^3_{x_1x_2y_2}H(0,0)+\frac1{4\pi}\pa^3_{x_1x_2x_2}\log V(0);\\
0&=&\pa^3_{\xi_2\xi_2\xi_2}\mcal F(0)=2\pa^3_{x_2x_2x_2}H(0,0)+6\pa^3_{x_2x_2y_2}H(0,0)+\frac1{4\pi}\pa^3_{x_1x_2x_2}\log V(0).
\end{eqnarray*}
Now, due to the harmonicity of the derivatives of $H$, summing the first and third line, and then the third and fourth line, gives:
\begin{eqnarray*}
\frac1{4\pi}\pa_{x_1}\D\log V(0)&=&-4\pa_{x_1}\l(\pa^2_{x_1y_1}H+\pa^2_{x_2y_2}H\r)(0,0),\\
\frac1{4\pi}\pa_{x_2}\D\log V(0)&=&-4\pa_{x_2}\l(\pa^2_{x_1y_1}H+\pa^2_{x_2y_2}H\r)(0,0);
\end{eqnarray*}
putting these equivalences in the definition \eqref{eta} gives \eqref{eta2}.\\
If $\O$ is simply connected, then it is well-known that the Robin function $R(\xi)=H(\xi,\xi)$ solves $-\D R=\frac2\pi e^{-4\pi R}$ (see for instance \cite{bf}), therefore for any $\xi\in\O$ one gets
\begin{eqnarray*}
16\pi^2\n_x\l(\pa^2_{x_1y_1}H+\pa^2_{x_2y_2}H\r)(\xi,\xi)&=&8\pi^2\n_x\D R(\xi)\\
&=&-16\pi\n_x\l(e^{-4\pi R(\xi)}\r)\\
&=&64\pi^2e^{-4\pi R(\xi)}\n_xR(\xi)\\
&=&-32\pi^3\D R(\xi)\n_xR(\xi)\\
&=&-64\pi^3\l(\pa^2_{x_1y_1}H(\xi,\xi)+\pa^2_{x_2y_2}H(\xi,\xi)\r)\n_xH(\xi,\xi),
\end{eqnarray*}
and in particular for $\xi=0$ one gets $\eta_0=0$.\\
On the other hand, if $\O$ is not simply connected, then $R$ does not solve the previous Liouville-type equation but rather a different PDE involving Bergman kernel (for details see for instance \cite{bf}, p. 211). Therefore, there exists $\xi\in\Omega$ such that
$$64\pi^3\l(\pa^2_{x_1y_1}H(\xi,\xi)+\pa^2_{x_2y_2}H(\xi,\xi)\r)\n_xH(\xi,\xi)+16\pi^2\n_x\l(\pa^2_{x_1y_1}H+\pa^2_{x_2y_2}H\r)(\xi,\xi)\ne0,$$
namely $\eta_0\not=0$. Now, we choose the potential $V$ so that such a point $\xi$ is a critical point of $\mcal F$ with the required order of degeneracy so that Definition \ref{ammis2} is verified. It is clear that the origin satisfies all the assumptions once we relabel $\O+\xi$ by $\O$.
\end{proof}\

In view of the previous considerations, Theorem \ref{molt} can be extended only to multiply connected domains as follows.

\begin{theorem}\label{molt2}$ $\\
Let $\Omega$ be a multiply connected domain and $V\in C^\infty\l(\ol\O\r)$ be a positive admissible potential of order $N$ (in the sense of Definition \ref{ammis2}).\\
If the equation
\begin{equation}\label{sc2}
\n\mcal P(\xi)=\eta_0
\end{equation}
has $K$ distinct stable solutions, then, there exist $\rho_0>0$ and $K$ families of solutions $\l\{u_\rho^i\r\}$, $i=1,\dots,K$, to \eqref{liou} for $\rho\in(0,\rho_0)$, all blowing up at $0$ as $\rho$ goes to $0$ (in the sense of Definition \ref{blowup}) and such that $u_\rho^i\not\equiv u_\rho^j$ if $i\ne j$.\\
In particular, this holds true if $\mcal F$ has exactly $K+1$ nodal lines.
\end{theorem}\

The only new tool in the proof of Theorem \ref{molt2} is the following generalization of Proposition \ref{eqxi}. Since we just need minor adaptations with respect to previous sections, proofs will be sketchy.

\begin{lemma}\label{eqxi2}$ $\\
Let $\xi$ satisfy $|\xi|=O\l(\rho^\frac2N\log^\frac1N\frac1\rho\r)$, $\phi_\xi$ be as in Lemma \ref{contr} and $\mcal P$ as in \eqref{p2}.\\
Then,
\begin{equation}\label{eq2}
\int_\O(\mcal R+\mcal L\phi+\mcal N(\phi))\mrm P\psi_i=\pa_{\xi_i}\mcal P(\xi)-\eta_i+O\l(\rho^2\r)
\end{equation}
\end{lemma}

\begin{proof}[Sketch of the proof]$ $\\
Most of the results from Sections 2, 3, 4, 5 hold in the same form with the following differences.\\
In Lemma \ref{de}, $\mcal E(x,\xi)$ verifies
$$\pa_{x_i}\mcal E(\xi,\xi)=\frac1{2\pi}\pa_{\xi}\mcal P(\xi)+O\l(|\xi|^{N+1}\r)=O\l(|\xi|^N\r)\q\q\q\q\q\q\mcal E(x,\xi)=O\l(|\xi|^N|x-\xi|+|x-\xi|^2\r).$$
Moreover, since now $|\mcal P(\xi)|=O\l(|\xi|^{N+1}\r)$, then Proposition \ref{rlp} states
$$\|\mcal R\|_p=O\l(\rho^{\frac2p-1}\l(\rho^2\log\frac1\rho+|\xi|^N\r)\r),$$
whereas in Lemma \ref{lin} one has
$$\|\phi\|=O\l(\rho^{\frac2p-1}\log\frac1\rho\l(\rho^2\log\frac1\rho+|\xi|^N\r)\r).$$
In Section 5 an equivalent of Lemma \ref{nophi} holds, still with $|\xi|^N$ in place of $|\xi|^2$; therefore, for $|\xi|=O\l(\rho^\frac2N\log^\frac1N\frac1\rho\r)$ the error in the statement is negligible. Finally, since \eqref{rx} still holds true with $\mcal P$ now being defined by \eqref{p2}, then \eqref{eq} can be proved just like in Proposition \ref{eqxi} and the claim is proved.
\end{proof}\

Before proving Theorem \ref{molt2} it is interesting to say a few words about the case $N=1$.

\begin{remark}$ $\\
In the case $N=1$, Lemma \ref{eqxi2} would give the formula:
$$\int_\O(\mcal R+\mcal L\phi+\mcal N(\phi))\mrm P\psi_i=4\pi^2\l(D^2\mcal F(0)\xi\r)_i-\eta_i+O\l(\rho^2\r).$$
If $D^2\mcal F(0)$ is invertible, the only zero to the leading order term is $\xi=\frac1{4\pi^2}\l(D^2\mcal F(0)\r)^{-1}\eta$.\\
Therefore, we would not get any multiplicity of solution, consistently with the uniqueness result proved by \cite{bjly2}.
\end{remark}\

\begin{proof}[Sketch of the proof of Theorem \ref{molt2}]$ $\\
Since the equation \eqref{eq2} has $K$ stable solutions, then due to Lemma \ref{eqxi2} we get $\xi^i$ with \eqref{eq2} vanishing, hence solutions $u^i$ to \eqref{liou} for $i=1,\dots,K$.\\
As in the proof of Theorem \ref{molt}, \eqref{ui} shows that the sequence is blowing up; then, one can prove $u^i\not\equiv u^j$ for $i\ne j$ by writing $\xi^i=\l(\rho^\frac2N\log^\frac1N\frac1\rho\r)\xi^i_0$ with $\frac1C\le\l|\xi^1_0-\xi^2_0\r|\le C$ and
\begin{eqnarray*}
\l|u^1\l(\xi^1_0\r)-u^2\l(\xi^1_0\r)\r|&=&\l|\log\frac1{\rho^4}-\log\frac1{\l(\rho^2+|\xi^2_0-\xi^1_0|^2\rho^\frac4N\log^\frac2N\frac1\rho\r)^2}\r|+O(1)\\
&=&2\log\l(\rho^{\frac4N-2}\log^\frac2N\frac1\rho\r)+O(1)\\
&\us{\rho\to0}\to&+\infty.
\end{eqnarray*}
To prove the last part, we use Lemma \ref{semplcon}, which ensures that $\eta_0\ne0$ for some $\xi\in\O$; up to translating, it will not be restrictive to assume that this occurs in $\xi=0$. Finally, thanks to Proposition \ref{grado}, we may also assume that \eqref{sc2} has $K$ stable solutions.
\end{proof}\

\appendix\

\section{Appendix: Degree computations}\

In this appendix we compute the degree of a generic non-degenerate polynomial $\mcal P$ like the ones in \eqref{p} and \eqref{p2}.\\
We believe that this result may be already known, as we consider rather well-studied objects, but we could not find any references.

\begin{proposition}\label{grado}$ $\\
Let $V$ be an admissible potential of order $N$ (in the sense of Definition \ref{ammis2}), let $\mcal P$ be defined by \eqref{p2} and let $M\le N+1$ be the number of nodal lines of $\mcal P(\xi)$.\\
Then, for any $\eta_0\in\R^2$ there exists $R>0$ such that
\begin{equation}\label{gradop}
\deg(\n\mcal P,B_R(0),\eta_0)=1-M.
\end{equation}
Moreover, for a.e. $\eta_0\in\R^2$ the equation $\n\mcal P=\eta_0$ has at least $|1-M|$ distinct stable solutions. If $N=2$ this holds true for any $\eta_0\ne0$.
\end{proposition}\

First of all, we show that it makes actually sense to compute the degree of any admissible $\mcal P$.

\begin{lemma}\label{bendef}$ $\\
Assume $V$ is admissible of order $N$.\\
Then, for any $\eta_0\ne0$ there exists $R>0$ such that any solution to $\n\mcal P(\xi)=\eta_0$ is contained in the open disk $B_R(0)$.\\
Moreover, for a.e. $\eta_0\in\R^2$ all the zeros $\xi$ to $\n\mcal P(\xi)-\eta_0$ are non-degenerate.
\end{lemma}

\begin{proof}$ $\\
In polar coordinates $(r,t)$ we can write $\xi=re^{\imath t}$, $\eta_0=r_0e^{\imath t_0}$ and $$\mcal P(\xi)-\langle\eta_0,\xi\rangle=r^{N+1}p(t)-rr_0\cos(t-t_0),$$
for some trigonometric polynomial $p$ of degree $N+1$. Since $\mcal P$ has $\xi=0$ as its only critical point, there will be no values of $t$ solving $p(t)=p'(t)=0$.\\
Solutions to $\n\mcal P=\eta_0$ verify
\begin{equation}\label{gradp}
\l\{\begin{array}{l}(N+1)r^Np(t)-r_0\cos(t-t_0)=0\\r^{N+1}p'(t)+rr_0\sin(t-t_0)=0\end{array}\r.,
\end{equation}
which yields
\begin{equation}\label{pt}
(N+1)^2p(t)^2+p'(t)^2=\l(\frac{r_0\cos(t-t_0)}{r^N}\r)^2+\l(-\frac{r_0\sin(t-t_0)}{r^N}\r)^2=\frac{r_0^2}{r^{2N}}.
\end{equation}
Since the left-hand side is always positive on $[0,2\pi]$, it will be bounded from below by some positive constant $C>0$, hence we deduce $r\le\l(\frac{r_0^2}C\r)^\frac1{2N}$ for any critical point.\\
To show the non-degeneracy, we see that the Hessian matrix on critical points is
$$\l(\begin{array}{cc}N(N+1)r^{N-1}p(t)&(N+1)r^Np'(t)+r_0\sin(t-t_0)\\(N+1)r^Np'(t)+r_0\sin(t-t_0)&r^{N+1}p''(t)+rr_0\cos(t-t_0)\end{array}\r).$$
Evaluating in the critical points satisfying \eqref{gradp}, the determinant equals
\begin{equation}\label{dethess}
Nr^{2N}\l((N+1)p(t)p''(t)-Np'(t)^2+(N+1)^2p(t)^2\r)=N(N+1)^2r^{2N}p(t)^\frac{2N+1}{N+1}\l(\l(p(t)^\frac1{N+1}\r)''+(N+1)^2p(t)^\frac1{N+1}\r).
\end{equation}
This can identically vanish only if $p(t)=C(\cos((N+1)t+\t))^{N+1}$ for some $C,\t$; however, if $p(t)$ had this form, then $\mcal P(\xi)$ would have non-zero critical points, hence $\mcal F$ would not satisfy the assumptions from Definition \ref{ammis}. Therefore, the Hessian determinant can have only a finite number of zeros.\\
On the other hand, by plugging \eqref{pt} in the first equation of \eqref{gradp} we deduce that any critical point satisfies
\begin{equation}\label{eqp}
(N+1)p(t)=\sqrt{(N+1)^2p(t)^2+p'(t)^2}\cos(t-t_0).
\end{equation}
For each zero $t$ to \eqref{dethess} the last equation equals zero for at most two values of $t_0$; since one has degenerate critical points only if $t$ is a zero to the Hessian determinant and $t_0$ solves \eqref{eqp}, this an occur for finitely many of $t_0$, that is only for $\eta_0$ in a negligible set of the plane.
\end{proof}\

To compute the degree of the polynomial $\mcal P$, we exploit a homotopical equivalence with some very well-known polynomials $\mcal Q_M$, given by the real part of complex powers.

\begin{lemma}\label{omot}$ $\\
Let $V$ be admissible of order $N$ such that $\mcal P(\xi)$ has $M\ge1$ distinct nodal lines and let $\mcal Q_M(\xi)$ be defined by
\begin{equation}\label{pm}
\mcal Q_M(\xi)=\Re\l((\xi_1+\imath\xi_2)^M\r)=\sum_{j=0}^{\l\lfloor\frac M2\r\rfloor}(-1)^j\bin{M}{2j}\xi_1^{M-2j}\xi_2^{2j}.
\end{equation}
Then, for any $R>0$ there exists a homotopical equivalence $F:[0,1]\x\R^2\to\R^2$ such that $F(0,\xi)=\n\mcal P(\xi)$, $F(1,\xi)=\n\mcal Q_M(\xi)$ and any solution to $F(s,\xi)=0$ is contained in the open disk $B_R(0)$ for $s\in[0,1]$.\\
In the case $M=0$ when $\mcal P(\xi)\ne0$ for any $\xi\ne0$, then the same holds true with
\begin{equation}\label{qpm}
\l\{\begin{array}{ll}\mcal Q_+(\xi)=|\xi|^2&\text{if }\mcal P(\xi)>0\,\fa\xi\ne0\\\mcal Q_-(\xi)=-|\xi|^2&\text{if }\mcal P(\xi)<0\,\fa\xi\ne0\end{array}\r..
\end{equation}
\end{lemma}

\begin{proof}$ $\\
In polar coordinates we can write
\begin{eqnarray*}
\mcal P(\xi)&=&r^{N+1}\wt p(t)\prod_{i=1}^M\sin(t-t_i)\\
\mcal Q_M(\xi)&=&r^M\cos(Mt)=r^M\frac{(-1)^{M-1}}{2^M}\prod_{i=1}^M\sin\l(t-\frac{2i-1}{2M}\pi\r),
\end{eqnarray*}
with $\wt p$ being a constantly-signed trigonometric polynomial and $0\le t_1<\dots<t_M<\pi$. If $\wt p$ has the same sign of $(-1)^{M-1}$ then a homotopical equivalence may be obtained by interpolating $F(s,\xi)=\n\mcal P^s(\xi)$, with $\mcal P^s$ given by
$$\mcal P^s(\xi)=\l((1-s)r^{N+1}+sr^M\r)\l((1-s)\wt p(t)+s\frac{(-1)^{M-1}}{2^M}\r)\prod_{i=1}^M\sin\l(t-t_{i,s}\r)\q\q\q t_{i,s}:=(1-s)t_i+s\frac{2i-1}{2M}\pi.$$
We suffice to show that, for any $s\in[0,1]$, the only solution to $F(s,\xi)$ is $\xi=0$. By deriving in $r,t$ we see that critical points of $\mcal P^s$ must satisfy
$$\l\{\begin{array}{l}\l((1-s)(N+1)r^N+sMr^{M-1}\r)\l((1-s)\wt p(t)+s\frac{(-1)^{M-1}}{2^M}\r)\prod_{i=1}^M\sin\l(t-t_{i,s}\r)=0\\\l((1-s)r^{N+1}+sr^M\r)\l((1-s)\wt p'(t)\prod_{i=1}^M\sin(t-t_{i,s})+\l((1-s)\wt p(t)+s\frac{(-1)^{M-1}}{2^M}\r)\pa_t\prod_{i=1}^M\sin(t-t_{i,s})\r)=0\end{array}\r..$$
Let us look at the first equation: since the first two factors do not change sign, it equals zero when $r=0$, corresponding to the solution $\xi=0$, or when $\sin(t-t_{i_0,s})=0$ for some $i_0$. If the latter condition is satisfied and not the former, then the second equation gives $\pa_t\prod_{i=1}^M\sin(t-t_{i,s})=0$. Since, from the first equation, one cannot have $\cos(t-t_{i_0,s})=0$, then it must be $\sin(t-t_{i,s})=0$ for $i\ne i_0$, but this is impossible because $0\le t_{1,s}<\dots<t_{M,s}<\pi$. We therefore excluded the case of critical points $\xi\ne0$.
On the other hand, if $\wt p$ and $(-1)^{M-1}$ have opposite sign, the same map is an equivalence between $\n\mcal P(\xi)$ and $-\n\mcal Q_M(\xi)=\n\mcal Q_M\l(e^{\imath\frac M\pi}\xi\r)$. To get a homotopical equivalence between the latter and $\n\mcal Q_M$ just consider a rotation $\wt F(s,\xi)=\n\mcal Q_M\l(e^{\imath\frac{(1-s)M}\pi}\xi\r)$.\\
A similar homotopical equivalence can be made in the case $M=0$, when one has $\mcal P(\xi)=r^{N+1}\wt p(t)$ and $\mcal Q_\pm(\xi)=r^2$: it suffices to take
$$F(s,\xi)=\n\mcal P^s(\xi)\q\q\q\q\q\q\mcal P^s(\xi)=(1-s)r^{N+1}\wt p(t)+sr^2.$$
\end{proof}\

The last step to prove Proposition \ref{grado} is the computation of the degree of $\mcal Q_M$:

\begin{lemma}\label{gradore}$ $\\
Let $M\ge1$ and $\eta_0\in\R^2$ be given and let $\mcal Q_M(\xi)$ be as in \eqref{pm}.\\
Then, for any $R>\l(\frac{|\eta_0|}M\r)^\frac1{M-1}$ it holds
$$\deg\l(\n\mcal Q_M,B_R(0),\eta_0\r)=1-M.$$
\end{lemma}

\begin{proof}$ $\\
We suffice to consider the case $\eta_0\ne0$ and then argue by approximation.\\
We look for solutions to $\n\mcal Q_M=\eta_0$ using polar coordinates, as in the proof of Lemma \ref{bendef} with $M$ in place of $N+1$ and $p(t)=\cos(Mt)$. Therefore \eqref{gradp} becomes
$$\l\{\begin{array}{l}Mr^{M-1}\cos(Mt)-r_0\cos(t-t_0)=0\\-Mr^M\sin(Mt)+rr_0\sin(t-t_0)=0\end{array}\r.$$
and, since $p'(t)=-M\sin(Mt)$, then \eqref{pt} becomes $M^2=\frac{r_0^2}{r^{2M-2}}$, namely $r=\l(\frac{r_0}M\r)^\frac1{M-1}$. Substituting in the previous equations gets
$$\l\{\begin{array}{l}r_0(\cos(Mt)-\cos(t-t_0))=0\\-\l(\frac{r_0^M}M\r)^\frac1{M-1}(\sin(Mt)-\sin(t-t_0))=0\end{array}\r.,$$
namely $t=\frac{-t_0+2i\pi}{M-1}$ for $i=1,\dots,M-1$.\\
At each of these points the Hessian determinant \eqref{dethess} equals $-(M-1)^2M^2r^{2M-2}<0$. Therefore, since all such points are contained in $B_R(0)=$ with $R>\l(\frac{r_0}M\r)^\frac1{M-1}=\l(\frac{|\eta_0|}M\r)^\frac1{M-1}$, one of the definition of degree gets $\deg(\mcal Q_M,B_R(0),\eta_0)=\sum_{i=1}^{M-1}(-1)=1-M$.
\end{proof}\

\begin{proof}[Proof of Proposition \ref{grado}]$ $\\
Thanks to Lemma \ref{bendef}, $\n\mcal P(\xi)\ne\eta_0$ for any $\xi\in\pa B_R(0)$ if $R$ is large enough, therefore the computation of the degree of $\mcal P$ makes sense.\\
Lemma \ref{omot} and the homotopy invariance of the degree imply that, for any $R>0$
$$\deg(\n\mcal P,B_R(0),0)=\deg(\n\mcal Q_M,B_R(0),0).$$
Again from Lemma \ref{bendef} and the properties of the degree, if $R$ is large enough then
$$\deg(\n\mcal P,B_R(0),0)=\deg(\n\mcal P,B_R(0),\eta_0),$$
whereas Lemma \ref{gradore} gives
$$\deg(\n\mcal Q_M,B_R(0),0)=1-M,$$
which proves \eqref{gradop}. In the case $M=0$, one has the homotopical equivalence with either $\n\mcal Q_+$ or $\n\mcal Q_-$, defined by \eqref{qpm}; since $\n\mcal Q_\pm(\xi)=\pm2\xi$, then the degree of each map equals $1$, therefore the formula still holds in this case.\\
Moreover, from the last statement of Lemma \ref{bendef}, $\mcal P(\xi)-\langle\eta_0,\xi\rangle$ is a Morse function for a.e. $\eta_0\in\R^2$, therefore for such values one gets at least $|1-M|$ different stable solutions.\\
Finally, in the case $N=2$ the number $M$ of zeros can be either $1$ or $3$ and in the former case there is nothing to prove. In the case $M=3$, since $\deg(\n\mcal P,B_R(0),\eta_0)\ne0$, there exists $\xi$ such that $\n\mcal P(\xi)=\eta_0$. Moreover, being $\n\mcal P$ even, one also has $\n\mcal P(-\xi)=\eta_0$; since $\n\mcal P(0)=0$, if $\eta_0\ne0$ one gets two different solutions $\xi\ne0\ne-\xi$.
\end{proof}\

\bibliography{batgropi_def}
\bibliographystyle{abbrv}

\end{document}